\documentclass[10pt]{amsart}

\usepackage{amssymb, mathdots,url}
\usepackage{braket}
\usepackage{verbatim}
\usepackage{mathtools}
\usepackage[usenames]{xcolor}
\usepackage[all]{xy}
\usepackage{paralist}
\usepackage{tikz-cd}
\usepackage[colorlinks,pagebackref=true]{hyperref}
\hypersetup{colorlinks,linkcolor={blue},citecolor={blue},urlcolor={red}}

\newcommand{\cat}[1]{\operatorname{\mathsf{#1}}}

\newcommand{\rmitem}[1]{\item[\text{\textup{(#1)}}]}

\newtheorem{theorem}{Theorem}[section]
\newtheorem{lemma}[theorem]{Lemma}
\newtheorem{proposition}[theorem]{Proposition}
\newtheorem{prop}[theorem]{Proposition}

\newtheorem{corollary}[theorem]{Corollary}

\theoremstyle{definition}

\newtheorem{definition}[theorem]{Definition}
\newtheorem{example}[theorem]{Example}
\theoremstyle{remark}
\newtheorem{remark}[theorem]{Remark}

\def\Hom{{\rm Hom}}

\newcommand{\Ai}{A_{{\rm idm}}}
\newcommand{\Bi}{B_{{\rm idm}}}
\def\RHom{{\rm RHom}}
\newcommand{\ot}{\otimes}
\newcommand{\op}{\mathrm{op}}

\title{Some homological aspects of idempotents in idempotented algebras}

\author{Arnab Mitra}
\address{Tata Institute of Fundamental Research, Mumbai, India}
\email{00.arnab.mitra@gmail.com}
\author{Rishi Vyas}
\email{vyas.rishi@gmail.com}

\date{\today}

\numberwithin{equation}{section}

\begin{document}

\setcounter{tocdepth}{1}

\begin{abstract}
We study special idempotents (as described by Bushnell and Kutzko) and split idempotents in the context of module and derived categories for idempotented algebras. We then characterize these concepts for path algebras of quivers. 
\end{abstract}

\maketitle
\tableofcontents

\section{Introduction}

An associative algebra $A$ is said to be idempotented if given any element $a\in A$, there exists an idempotent $e\in A$ (i.e. $e^{2}=e$) such that $ae=ea=a$, and given any two idempotents $e_{1},e_{2}\in A$, there exists an idempotent $e\in A$ such that $ee_{i}=e_{i}e=e_{i}$ for $i=1,2$. Every unital algebra is idempotented, but non-unital examples do appear in natural contexts.

The notion of a {\it special idempotent} was introduced in \cite{BK} for Hecke algebras of $p$-adic reductive groups, which form a prominent class of non-unital idempotented algebras (see \S \ref{ss:red_grp} for details). These idempotents play a fundamental role in the {\it theory of types} and are objects of interest in the representation theory of $p$-adic groups. In \cite{SZ}, analogues of results in \cite{BK} for special idempotents  were obtained for the Harish-Chandra Schwartz algebra of a $p$-adic reductive group. Given their importance, it is natural to study these idempotents in greater generality and develop their formalism: this is one of the purposes of this note. In \S \ref{sim}, we study special idempotents for arbitrary idempotented algebras: an idempotent $e\in A$ is special if the category $\cat{M}(A)_{e}:= \{M\in \cat{M}(A)\ \vert \ AeM = M\}$ is closed under subquotients, where $\cat{M}(A)$ is the category of non-degenerate $A$-modules. In particular, we generalize some results from \cite{BK} to module categories over such algebras.

 In \S \ref{sid} we introduce and study the notion of a \emph{split idempotent} (see  Definition \ref{defn:split}): this notion isolates idempotents which are well behaved from a certain homological perspective. Derived techniques have become increasingly important in the representation theory of $p$-adic groups, especially in the context of modular representations (see \cite{Ha} for an exposition). Using our results on special and split idempotents, we initiate a preliminary investigation of the derived category of an idempotented algebra. In  \S \ref{s:model_structure} we show that every complex is quasi-isomorphic to both a K-projective and a K-injective one. Using this, in \S \ref{sid}, we obtain the derived versions of the results  in \S \ref{sim}. 
 
 When $e\in A$ is split, $\cat{M}(A)_{e}$ appears as a direct summand of $\cat{M}(A)$ (see Proposition \ref{prop:30}). In \S \ref{ss:decom}, we isolate the conditions on a family of special and split idempotents $\{e_{i}\}$ that ensure a decomposition of the (derived) category of non-degenerate  $A$-modules into a product of the categories $\cat{M}(A)_{e_{i}}$ (see Theorem \ref{prop:896} and Theorem \ref{prop:898}).

As evidenced by the case of Hecke algebras, it can be a difficult problem to classify special and split idempotents for a given idempotented algebra. Path algebras of quivers form a natural class of idempotented algebras. In \S \ref{ss:patha} we achieve a complete classification of these idempotents for such algebras. The proofs of the results in this section are elementary, but of a technical combinatorial nature.  

To obtain our results in \S \ref{s:spc_idm} and \S \ref{ss:patha}, we require some basic material on non-degenerate modules over idempotented algebras. We obtain these results in \S \ref{sect_prelim}. Module categories over idempotented algebras have been studied on several occasions, especially in the context of the representation theory of $p$-adic groups (see \cite{Ber}, \cite{BW}, or \cite{DK}). Many of the results in \S \ref{sect_prelim} are standard for unital rings, and may even be known to experts in the idempotented case. However, as we have not been able to find explicit references we believe it worthwhile to have them formally written down.

\subsection*{Notation and conventions}\label{ss:not_con}

All algebras will be defined over a fixed, unital, commutative base ring $\mathbb{K}$. An unadorned $\otimes$ will mean tensoring over $\mathbb{K}$. The terms `ring', `module', and `complex' used without any prefix will always mean over $\mathbb{K}$. All complexes are indexed cohomologically.

Recall that an associative algebra $A$ is said to be idempotented if given any element $a\in A$, there exists an idempotent $e\in A$ (i.e. $e^{2}=e$) such that $ae=ea=a$, and given any two idempotents $e_{1},e_{2}\in A$, there exists an idempotent $e\in A$ such that $ee_{i}=e_{i}e=e_{i}$ for $i=1,2$. The reader may convince themselves that $A$ is idempotented if and only if  given elements $a_{1},\ldots,a_{n} \in A$, there exists an idempotent $e\in A$ such that $ea_{i} = a_{i}e = a_{i}$ for all $i\in \{1,\ldots,n\}$. In this note, $A$ and $B$ will denote idempotented algebras.  We use $\Ai$ to represent the set of all idempotents in $A$.  A morphism between two idempotented algebras means a morphism between the underlying associative algebras: a morphism of unital algebras viewed as idempotented algebras is not necessarily unital. 

A left module $M$ over  $A$ is called \emph{non-degenerate} if for any given $m\in M$, there is some $e\in \Ai$ such that $em=m$. We use $\cat{M}(A)$ to represent the category of non-degenerate left $A$-modules. The category of complexes of objects of $\cat{M}(A)$ will be denoted by $\cat{C}(A)$, with associated homotopy category $\cat{K}(A)$ and derived category $\cat{D}(A)$. The category of all left $A$-modules is denoted by $\overline{\cat{M}}(A)$:  $\cat{M}(A)$ is a full abelian subcategory of $\overline{\cat{M}}(A)$. The category of complexes of objects of $\overline{\cat{M}}(A)$ is denoted by $\overline{\cat{C}}(A)$, with associated homotopy category $\overline{\cat{K}}(A)$ and derived category $\overline{\cat{D}}(A)$. 

Both $A^{\mathrm{op}}$  and $A\otimes B^{\mathrm{op}}$ are idempotented algebras. Right modules will be dealt with using the formalism of the opposite ring. Left $A$ and right $B$ bimodules are identified with left modules over $A\otimes B^{\mathrm{op}}$; such a bimodule is non-degenerate on both sides if and only if it is non-degenerate as an $A\otimes B^{\mathrm{op}}$-module. 

If $M\in \overline{\cat{C}}(A)$, we use $\operatorname{H}^{i}(M)$ for its $i^{\mathrm{th}}$ cohomology $A$-module, and $\operatorname{Z}^{i}(M)$ for the $A$-module of $i^{\mathrm{th}}$ cocycles. If $M$ and $N$ are objects of $\overline{\cat{M}}(A)$, we use $\mathrm{Hom}_{A}(M,N)$ to denote the module of $A$-linear maps between them. More generally, if $M$ and $N$ are objects of $\overline{\cat{C}}(A)$, we use $\mathrm{Hom}_{A}(M,N)$ for the complex of morphisms between them, where \[\Hom_{A}(M,N)^{k}=\prod_{i\in \mathbb Z}\Hom_{A}(M^{i},N^{i+k})\] and $\partial(f)=\partial_{N}\circ f-(-1)^{k}f\circ\partial_{M}$ for $f\in \Hom_{A}(M,N)^{k}.$ If $M \in \overline{\cat{C}}(A^{\mathrm{op}})$ and $N\in \overline{\cat{C}}(A)$, $M\otimes_{A} N$ is the complex with $k^{\mathrm{th}}$ term \[
(M\otimes_{A} N)^{k}=\bigoplus_{i+j=k}M^{i}\otimes N^{j}
\]
 and $\partial^{k}(m\otimes n)=\partial(m)\otimes n+ (-1)^{i}m \otimes \partial(n)$ where $m\in M^{i},n\in N^{j}$. For morphisms in $\overline{\cat{C}}(A)$ or $\overline{\cat{K}}(A)$, we use $\mathrm{Hom}_{\overline{\cat{C}}(A)}(M,N)$ and $\mathrm{Hom}_{\overline{\cat{K}}(A)}(M,N)$, respectively. Recall that $\mathrm{Hom}_{\overline{\cat{C}}(A)}(M,N) = \operatorname{H}^{0}(\mathrm{Hom}_{A}(M,N))$ and $\mathrm{Hom}_{\overline{\cat{K}}(A)}(M,N) = \operatorname{Z}^{0}(\mathrm{Hom}_{A}(M,N))$.

An object  $P\in \cat{K}(A)$ is called \emph{K-projective} if $\mathrm{Hom}_{A}(P,N)$ is acyclic for all acyclic complexes $N$ in  $\cat{K}(A)$. Similarly, an object $I\in \cat{K}(A)$ is called \emph{K-injective} if $\mathrm{Hom}_{A}(N,I)$ is acyclic for all acyclic complexes $N$ in  $\cat{K}(A)$. There are similar notions of K-projectivity and K-injectivity for $\overline{\cat{K}}(A)$, where we test over all acyclic complexes in $\overline{\cat{K}}(A)$. The  K-projectivity or K-injectivity of a complex \emph{does} depend on whether we are viewing it as an object of $\overline{\cat{K}}(A)$ or $\cat{K}(A)$. See Lemma \ref{cor:201}.

The notation $\mathrm{RHom}_{A}$ is used for the right derived functor of $\mathrm{Hom}_{A}$, and $\otimes_{A}^{\mathrm{L}}$ is used for the left derived functor of $\otimes_{A}$. The construction of $\ot_{A}^{\mathrm{L}}$ and $\mathrm{RHom}_{A}$ is independent of whether we are constructing them in $\cat{D}(A)$ or $\overline{\cat{D}}(A)$ (Lemma \ref{cor:201}).

\subsection*{Acknowledgements} \label{ack}
Both authors would like to thank the Tata Institute of Fundamental Research, Mumbai, for providing an effective environment in which a part of this work was done, and Sandeep Varma for his suggestions.

\section{Preliminaries on module and derived categories}\label{sect_prelim}
\subsection{Module categories of idempotented algebras}\label{s:pre_lem}

\begin{definition}
Define the functor \[\nu_{A}: \overline{\cat{M}}(A) \to \cat{M}(A)\] by \[\nu_{A}(M) := \{m\in M\ \vert \ \exists\  e\in \Ai \ \text{s.t.} \ em = m\}.\]
\end{definition}

Since $A$ is idempotented, $\nu_{A}(M)$ is an $A$-submodule of $M$ for any $M\in \overline{\cat{M}}(A)$. The functor $\nu_{A}$ is right adjoint to the canonical inclusion of $\cat{M}(A)$ in $\overline{\cat{M}}(A)$. 

The following lemma is standard (for example, see \cite[Lemma 0.12]{DK}).

\begin{lemma} \label{lem:101}
The functor $\nu_{A}: \overline{\cat{M}}(A)\to \cat{M}(A)$ is exact, and commutes with arbitrary direct sums.  
\qed
\end{lemma}

\begin{lemma} \label{lem:04}\label{lem:pre_5}
 Let $e\in \Ai$ and $M\in \overline{\cat{M}}(A)$. Then, there are the following maps:
 \begin{enumerate}
 \rmitem{i} The correspondence $f \mapsto f(e)$ establishes an isomorphism of modules \[\mathrm{Hom}_{A}(Ae,M)\cong eM.\]
 \rmitem{ii} For $M=Ae$, the isomorphism in $({\rm i})$ induces an isomorphism of rings \[\mathrm{End}_{A}(Ae)\cong eAe^{\mathrm{op}}.\]
\rmitem{iii} The correspondence $ea\otimes m \mapsto eam$ establishes an isomorphism of modules 
\[eA\otimes_{A} M\cong eM.
\]
 \end{enumerate}
\end{lemma} 
\begin{proof}
The proofs of (i) and (ii) are identical to those of  the corresponding statements in the unital case - see \cite[Lemma 1.3.3]{Ben95}.

For part (iii) consider the map of modules $\psi_{e}: eA\otimes_{A} M\to eM$ taking $ea\otimes m$ to $eam$. Clearly it is surjective. Note that $\sum_{i=1}^{n}(ea_{i}\otimes x_{i})=e\otimes(\sum_{i=1}^{n}ea_{i}x_{i})$, which implies that the map $\psi_{e}$ is injective as well.
\end{proof}

\begin{lemma}\label{lem:pre_2}
Let $M\in \cat{M}(A)$ and $N\in \overline{\cat{M}}(A)$, and let $\psi:M\to N$. If the induced maps $\operatorname{id}_{eA}\otimes_{A} \psi: eA\otimes_{A}M\to eA\otimes_{A}N$ are injective for all $e\in \Ai$, then the map $\psi$ is injective.
\end{lemma}
\begin{proof}
The result follows from Lemma \ref{lem:pre_5}, as $M = \bigcup_{e\in \Ai} eM$.
\end{proof}

\begin{lemma} \label{lem:01}
Let $M\in \overline{\cat{M}(}A)$ and $N\in \cat{M}(A)$, and let $\psi:M\to N$. If the induced maps  $\mathrm{Hom}_{A}(\operatorname{id}_{Ae}, \psi): \mathrm{Hom}_{A}(Ae,M)\to \mathrm{Hom}_{A}(Ae,N)$ are surjective for all $e\in \Ai$, then the map $\psi$ is surjective.
\end{lemma}
\begin{proof}
By the hypothesis and Lemma \ref{lem:04} the restriction of $f$ to $eM$ is a surjective map from $eM$ to $eN$ for each $e\in \Ai$. As $N = \bigcup_{e\in \Ai} eN$, the result follows.
\end{proof}

\begin{lemma}\label{lem:pre_6}
Let $e\in \Ai$ and $M\in \overline{\cat{M}}(A)$. Then, there is a natural isomorphism of modules: \[eA\otimes_{A} \nu_{A}(M)\cong eA\otimes_{A} M.\]
\end{lemma}
\begin{proof}
By Lemma \ref{lem:pre_5}, the canonical map $eA\otimes_{A} \nu_{A}(M)\to eA\otimes_{A} M$ corresponds to the inclusion $e\nu_{A}(M)\to eM$. Since $eM$ is a subset of  $\nu_{A}(M)$, $e\nu_{A}(M) = eM$.
\end{proof}

\begin{lemma}\label{lem:pre_4}
Let $L\in \overline{\cat{M}}(A\otimes B^{\mathrm{op}})$ and $N\in \cat{M}(B)$. Then, there is a natural isomorphism in $\cat{M}(A)$:
\[
\nu_{A}(L)\otimes_{B} N\cong \nu_{A}(L \otimes_{B} N).
\]
\end{lemma}
\begin{proof}
Let $\psi:\nu_{A}(L)\otimes_{B} N\to L \otimes_{B} N$ be the natural $A$-linear map given by $\psi(l\otimes n)=l\otimes n$ for $l\in L, n\in N$. 

The image of $\psi$ is contained in $\nu_{A}(L \otimes_{B} N)$. Indeed, if $x=\sum_{i=1}^{n}l_{i}\otimes n_{i}$  is an arbitrary element of $\nu_{A}(L)\otimes_{B} N$, then there exists an $e\in \Ai$ such that $el_{i}=l_{i}e=l_{i}$ for every $i$, and therefore that $ex=x$.

Next we show that $\psi$ is surjective. Let $x=\sum_{i=1}^{n}l_{i}\otimes n_{i}$  be an arbitrary element of $\nu_{A}(L\otimes_{B} N)$. Let $e\in \Ai$ be such that $ex=x$. Note that $el_{i}\in \nu_{A}(L)$ for all $i$ and $\psi(\sum_{i=1}^{n}el_{i}\otimes n_{i})=x$; this proves the surjectivity of $\psi$.

Finally we show that $\psi$ is injective. Let $e\in \Ai$. Consider the induced map 
\[
\psi_{e}: eA\otimes_{A}\nu_{A}(L) \otimes_{B} N \to eA \otimes_{A}\nu_{A}(L\otimes_{B} N).
\]
By Lemma \ref{lem:pre_6} the map $\psi_{e}$ is an isomorphism. The injectivity of $\psi$ now follows from Lemma \ref{lem:pre_2}.
\end{proof}

\begin{lemma}\label{lem:pre_1}
Let $\phi:A\to B$ be a morphism of algebras, and let $M\in \cat{M}(A)$. For $e_{1}, e_{2}\in\Ai$ such that $e_{i} m=m$ for $i=1,2$, $\phi(e_{1})\otimes m=\phi(e_{2})\otimes m$ in $B\otimes_{A} M$.
\end{lemma}
\begin{proof}
There exists an $e\in \Bi$ such that $e\phi(e_{i})=\phi(e_{i})e=\phi(e_{i})$ for $i=1,2$. We have
$ \phi(e_{1})\otimes m=e\phi(e_{1})\otimes m= e\otimes e_{1}m=e\otimes e_{2}m=\phi(e_{2})\otimes m.$
\end{proof}

\begin{lemma}\label{lem:pre_0}
Let $\phi:A\to B$ be a morphism of algebras, and let $M\in \cat{M}(A)$. Given $m\in M$ and $e\in \Ai$ such that $e m=m$, define the map $\iota_{M}: M\to B\otimes_{A} M$ by $\iota_{M}(m)=\phi(e)\otimes m$. Then, the map $\iota_{M}$ is independent of the choice of $e$, and is $A$-linear.
\end{lemma}
\begin{proof}
By Lemma \ref{lem:pre_1}, $\iota_{M}$ is independent of our choice of $e$. The fact that $\iota_{M}$ is $A$-linear is straightforward to show.
\end{proof}

The following lemma appears in \cite[Lemma XII.0.3]{BW}; we give a proof for the sake of completeness.

\begin{lemma}\label{lem:pre_3}
Let $M\in \overline{\cat{M}}(A)$. Then, the following are equivalent:
\begin{enumerate}
\rmitem{a.} $M\in \cat{M}(A)$.
\rmitem{b.} $A\otimes_{A}M \cong M$.
\rmitem{c.}$\nu_{A}(\Hom_{A}(A,M))\cong M$. 
\end{enumerate}
\end{lemma}
\begin{proof}
First suppose that $M\in \cat{M}(A)$. 

We begin by showing (b.). Consider the map $\iota=\iota_{M}$ defined in Lemma \ref{lem:pre_0} for $A=B$ and $\phi$ the identity map. By the aforementioned lemma the map is a well defined $A$-linear map. The injectivity of the map $\iota$ follows from the Lemmas \ref{lem:pre_5} and \ref{lem:pre_2}. Let $a\otimes m\in A\otimes_{A}M$, and let $e\in \Ai$ be such that $ea=ae=a$. Then $\iota(am)=e\otimes am=a\otimes m$.

Next we prove (c.). Let $f\in \nu_{A}(\Hom_{A}(A,M))$. There exists an $e_{f}\in \Ai$ such that $e_{f}f=f$. Define $T:\nu_{A}(\Hom_{A}(A,M))\to M$ by $T(f)=f(e_{f})$. It is straightforward to check that $T$ is a well defined $A$-linear map and is injective. We now show that $T$ is surjective. Given $m\in M$, define $f_{m}\in \Hom_{A}(A,M)$ by $f_{m}(a):=am$. If $e\in \Ai$ is such that $em=m$, then clearly $ef_{m}=f_{m}$ and thus $f_{m}\in \nu_{A}(\Hom_{A}(A,M))$. Now $T(f_{m})=f_{m}(e)=em=m$.

If either (b.) or (c.) holds, it is obvious that $M\in \cat{M}(A)$. 
\end{proof}

\begin{prop} \label{prop:adm_1}
Let $M\in \cat{M}(A)$, $N\in \cat{M}(B)$ and $L\in \overline{\cat{M}}(B\ot A^{\op})$. Then, there is a natural isomorphism of modules:
\[
\Hom_{B}(\nu_{B}(L)\otimes_{A} M,N)\cong \Hom_{A}\big(M, \nu_{A}\big(\Hom_{B}(\nu_{B}(L),N)\big)\big).
\]
\end{prop}
\begin{proof}
Given a $B$-linear map $f:\nu_{B}(L)\otimes_{A} M\to N$, define 
\[
\hat{f}\in  \Hom_{A}\big(M, \nu_{A}\big(\Hom_{B}(\nu_{B}(L),N)\big)\big)
\]
by $\hat{f}(m)(l)=f(l\otimes m)$, for $l\in L$ and $ m\in M$. Since $M\in \cat{M}(A)$, the image of $\hat{f}$ lies in $ \nu_{A}\big(\Hom_{B}(\nu_{B}(L),N)\big)$. On the other hand given an $A$-linear map $g:M \to \nu_{A}\big(\Hom_{\cat{M}(B)}(\nu_{B}(L),N)\big)$, define $\hat{g}\in \Hom_{\cat{M}(B)}(\nu_{B}(L)\otimes_{A} M,N)$ by $\hat{g}(l\otimes m)=g(m)(l)$, for $l\in L$ and $ m\in M$. It is easy to see that the correspondences $f\mapsto \hat{f}$ and $g\mapsto \hat{g}$ are inverse to each other.
\end{proof}

Analogous to the case of module categories over unital rings, we now obtain the following adjunction relations from Proposition \ref{prop:adm_1}.
\begin{corollary}\label{corr:adm_1}
Let $\phi:A\to B$ be a morphism of algebras. Let $M\in \cat{M}(A)$ and $N\in \cat{M}(B)$. Then, there are the following natural isomorphisms of modules:
\begin{enumerate}
\rmitem{i} $\Hom_{A}(\nu_{A}(N),M)\cong\Hom_{B}(N,\nu_{B}(\Hom_{A}(\nu_{A}(B),M)))$.
\rmitem{ii} $\Hom_{B}(B\otimes_{A} M,N)\cong\Hom_{A}(M,\nu_{A}(N))$.
\end{enumerate}
\end{corollary}
\begin{proof}
We prove (i). Interchanging the role of the algebras $A$ and $B$ in Proposition \ref{prop:adm_1} and considering $B$ as an $A\otimes B^{{\rm op}}$-module, we get that 
\begin{equation}\label{eq:adm_1}
\Hom_{A}(\nu_{A}(B)\otimes_{B} N,M)\cong \Hom_{B}\big(N, \nu_{B}\big(\Hom_{A}(\nu_{A}(B),M)\big)\big).
\end{equation}
By Lemma \ref{lem:pre_4} and Lemma \ref{lem:pre_3} the left hand side of the isomorphism in (\ref{eq:adm_1}) is isomorphic to $\Hom_{A}(\nu_{A}(N),M)$. This finishes the proof of (i).

Next we prove (ii). Considering $B$ as a $B\otimes A^{{\rm op}}$-module and appealing to Proposition \ref{prop:adm_1}, we get that
\[
\Hom_{B}(B\otimes_{A} M,N)\cong \Hom_{A}\big(M, \nu_{A}\big(\Hom_{B}(B,N)\big)\big).
\]
The statement now follows from Lemma \ref{lem:pre_3}.
\end{proof}

\subsection{Resolutions} \label{s:model_structure}

Every object in the category of complexes over a unital ring admits a quasi-isomorphism from a K-projective complex, and a quasi-isomorphism to a K-injective one. In this subsection, we show that the same holds for complexes of non-degenerate modules over an idempotented algebra.

\begin{lemma} \label{l:401}
In the category $\cat{M}(A)$, direct sums and direct products are exact.
\end{lemma}
\begin{proof}
Let $\{M_{i}\}_{i\in I}$ be a family of objects in $\cat{M}(A)$. The standard direct sum of this family, $\bigoplus_{i\in I} M_{i}$, belongs to $\cat{M}(A)$: it is therefore the direct sum in $\cat{M}(A)$. Thus direct sums are exact.

It is not necessarily the case that $\prod_{i\in I}M_{i}$ lies in $\cat{M}(A)$; instead, the product of the family $\{M_{i}\}_{i\in I}$ in $\cat{M}(A)$ is $\nu_{A}(\prod_{i\in I}M_{i})$. By Lemma \ref{lem:101}, $\nu_{A}$ is exact. It follows that direct products are also exact in $\cat{M}(A)$.
\end{proof}

Fix an injective cogenerator $\mathbb{K}^{*}$ for the category $\cat{M}(\mathbb{K})$.

\begin{lemma} \label{l:402}
The following hold:
\begin{enumerate}
\rmitem{i} $\bigoplus_{e\in \Ai} Ae$ is a projective generator for the category $\cat{M}(A)$.
\rmitem{ii} $\nu_{A}(\mathrm{Hom}_{\mathbb{K}}(A,\mathbb{K}^{*}))$ is an injective cogenerator for the category $\cat{M}(A)$.
\end{enumerate}
\end{lemma}
\begin{proof}
By Lemma \ref{lem:04}, $Ae$ is a projective object in $\cat{M}(A)$ for every $e\in \Ai$. Therefore $\bigoplus_{e\in \Ai} Ae$  is a projective object in $\cat{M}(A)$. Let $M\in \cat{M}(A)$, and suppose $\mathrm{Hom}_{A}(\bigoplus_{e\in \Ai} Ae,M)=0$. By Lemma \ref{lem:04}, this implies that $eM=0$ for $e\in \Ai$, and thus that $M=0$. This proves (i).

Let $M\in \cat{M}(A)$. By Proposition \ref{prop:adm_1}, there is an isomorphism of modules $\mathrm{Hom}_{\mathbb{K}}(A\ot_{A} M, \mathbb{K}^{*})\cong \mathrm{Hom}_{A}(M,\nu_{A}(\mathrm{Hom}_{\mathbb{K}}(A, \mathbb{K}^{*})))$. By Lemma \ref{lem:pre_3}, however, $A\otimes_{A} M\cong M$. It follows that $\nu_{A}(\mathrm{Hom}_{\mathbb{K}}(A,\mathbb{K}^{*}))$ is an injective cogenerator in $\cat{M}(A)$, as $\mathbb{K}^{*}$ is an injective cogenerator in $\cat{M}(\mathbb{K})$.
\end{proof}

\begin{prop} \label{prop:501}
The following hold:
\begin{enumerate}
\rmitem{i} Every complex in $\cat{K}(A)$ admits a quasi-isomorphism to a K-injective complex.
\rmitem{ii} Every complex in $\cat{K}(A)$ admits a quasi-isomorphism from a K-projective complex.
\end{enumerate}
\end{prop}
\begin{proof}
By Lemma \ref{l:401}  the category $\cat{M}(A)$ has exact direct products and sums; by Lemma \ref{l:402}, it admits a projective generator and an injective cogenerator. We can therefore appeal to \cite[Section 5.1]{Kr} for (i), and \cite[Section 5.2]{Kr} for (ii).
\end{proof}

\begin{remark}
There is a certain asymmetry present in the statements of Lemma \ref{lem:pre_3}, Corollary \ref{corr:adm_1}, and Lemma \ref{l:402}. This can be explained as follows. Quillen, in an unpublished note \cite{Q}, considers potential definitions for a suitable category of modules over a non-unital ring. Briefly, two possible options are as follows. Let $S$ be an associative ring, and suppose $S$ sits as an ideal in a unital ring $R$. Quillen defines the class of \emph{firm} modules to be those $M\in \cat{M}(R)$ such that the canonical map $S\otimes_{R}M\to M$ is an isomorphism, and the class of \emph{closed} modules to be those $M\in \cat{M}(R)$ such that the map $M\to \Hom_{R}(S,M)$ is an isomorphism. In \cite{Q}, Quillen shows that for idempotent $S$ (i.e. $S=S^{2}$), these categories do not depend on the choice of ambient ring $R$ and are equivalent to each other. 

The equivalence of (a.) and (b.) in Lemma \ref{lem:pre_3} shows that for an idempotented algebra $A$, the category $\cat{M}(A)$ coincides with the class of firm $A$-modules. If we were to instead work with the class of closed $A$-modules, we would get an analogous theory.
\end{remark}

\subsection{Derived categories of idempotented algebras}\label{s:adj_der}
We now obtain derived analogues of the results in \S \ref{s:pre_lem}. It is easy to see that the canonical functors $\cat{C}(A)\to \overline{\cat{C}}(A)$ and $\cat{K}(A)\to \overline{\cat{K}}(A)$ are fully faithful. The functor $\nu_{A}$ naturally extends to functors from $\overline{\cat{C}}(A)\to \cat{C}(A)$ and $\overline{\cat{K}}(A)\to \cat{K}(A)$; in both cases, these functors are right adjoint to the corresponding inclusion functors. 

\begin{lemma} \label{cor:201}
The following hold:
\begin{enumerate}
\rmitem{i} If $P$ is a K-projective object in $\cat{K}(A)$, then it is K-projective in $\overline{\cat{K}}(A)$.
\rmitem{ii} If $I$ is a K-injective object in $\overline{\cat{K}}(A)$, then $\nu_{A}(I)$ is K-injective in $\cat{K}(A)$.
\end{enumerate}
\end{lemma}
\begin{proof}
We prove (i). Let $P$ be a K-projective complex in $\cat{K}(A)$, and suppose $N$ is an acyclic complex in $\overline{\cat{K}}(A)$.  Observe that $\nu_{A}(N)$ is an object in $\cat{K}(A)$; by Lemma \ref{lem:101} it is acyclic. Since $\nu_{A}$ is right adjoint to the inclusion of $\cat{K}(A)$ in $\overline{\cat{K}}(A)$ there is an isomorphism of modules $\mathrm{Hom}_{\overline{\cat{K}}(A)}(P,N)\cong \mathrm{Hom}_{\cat{K}(A)}(P,\nu_{A}(N))$. Since $P$ is K-projective in $\cat{K}(A)$ and $\nu_{A}(N)$ is acyclic, $\mathrm{Hom}_{\cat{K}(A)}(P,\nu_{A}(N)) = 0$. The claim follows.
The proof of (ii) is similar. 
\end{proof}

By Lemma \ref{cor:201}, it follows that the canonical functor  $\cat{D}(A)\to \overline{\cat{D}}(A)$  is fully faithful as well. Since $\nu_{A}$ is exact, it admits an extension from  $\overline{\cat{D}}(A)\to \cat{D}(A)$, which we again denote by $\nu_{A}$. Again by Lemma \ref{cor:201}, $\nu_{A}: \overline{\cat{D}}(A)\to \cat{D}(A)$ is right adjoint to the inclusion of $\cat{D}(A)$ in $\overline{\cat{D}}(A)$.

\begin{prop} \label{prop:adm_1_der}
Let $M\in \cat{C}(A)$, $N\in \cat{C}(B)$ and $L\in \overline{\cat{C}}(B\otimes A^{{\rm op}})$. Then, there is a natural isomorphism of complexes: 
\[
\Hom_{B}(\nu_{B}(L)\otimes_{A}M,N)\cong \Hom_{A}\big(M, \nu_{A}\big(\Hom_{B}(\nu_{B}(L),N)\big)\big). 
\]
This induces a natural isomorphism of modules:
\[
\Hom_{\cat{C}(B)}(\nu_{B}(L)\otimes_{A}M,N)\cong \Hom_{\cat{C}(A)}\big(M, \nu_{A}\big(\Hom_{B}(\nu_{B}(L),N)\big)\big).
\]
\end{prop}
\begin{proof}
Let $k\in \mathbb{Z}$. Consider the following chain of isomorphisms:
\[
\begin{split}
\Hom_{B}(\nu_{B}(L)\otimes_{A} M, N)^{k}& \cong  \prod_{i\in \mathbb Z}\Hom_{B}((\nu_{B}(L)\otimes_{A} M)^{i}, N^{i+k})\\
& \cong \prod_{i\in \mathbb Z}\Hom_{B}(\bigoplus_{j+l=i}\nu_{B}(L)^{j}\otimes_{A} M^{l}, N^{i+k})\\
& \cong  \prod_{i\in \mathbb Z}\prod_{j+l=i}\Hom_{B}(\nu_{B}(L)^{j}\otimes_{A} M^{l}, N^{j+l+k})\\
& \cong  \prod_{i\in \mathbb Z}\prod_{j+l=i}\Hom_{A}(M^{l}, \nu_{A}(\Hom_{B}(\nu_{B}(L)^{j},N^{j+l+k})))\\
& \cong  \prod_{l\in \mathbb Z}\Hom_{A}(M^{l},\nu_{A}( \prod_{j\in  \mathbb{Z}}\nu_{A}(\Hom_{B}(\nu_{B}(L)^{j},N^{j+l+k}))))\\
& \cong  \prod_{l\in \mathbb Z}\Hom_{A}(M^{l},\nu_{A}( \prod_{j\in \mathbb{Z}}(\Hom_{B}(\nu_{B}(L)^{j},N^{j+l+k}))))\\
& \cong \prod_{\l\in \mathbb Z}\Hom_{A}(M^{l}, \nu_{A}(\Hom_{B}(\nu_{B}(L),N)^{l+k}))\\
& \cong \Hom_{A}(M, \Hom_{B}(\nu_{A}(L), N))^{k},
\end{split}
\]
where the fourth isomorphism follows from Proposition \ref{prop:adm_1}. We leave it to the reader to verify that the differentials of the above complexes are compatible with the isomorphism we have just established. This is tedious, but standard.
Finally, recall that for $M',N'\in \cat{C}(A)$ there is an isomorphism \[\Hom_{\cat{C}(A)}(M',N') \cong \operatorname{Z}^{0}(\Hom_{A}(M',N')). \qedhere \] 
\end{proof}

\begin{corollary}\label{cor:71}
Let $M\in \cat{K}(A)$, $N\in \cat{K}(B)$ and $L\in \overline{\cat{K}}(B\otimes A^{{\rm op}})$. Then, there is a natural isomorphism of modules:
\[
\Hom_{\cat{K}(B)}(\nu_{B}(L)\otimes_{A}M,N)\cong \Hom_{\cat{K}(A)}\big(M, \nu_{A}\big(\Hom_{B}(\nu_{B}(L),N)\big)\big).
\]
\end{corollary}
\begin{proof}
Given complexes $M'$ and $N'$ in $\cat{K}(A)$, there is a canonical isomorphism \[\Hom_{\cat{K}(A)}(M',N') \cong \operatorname{H}^{0}(\Hom_{A}(M',N')).\] Using the isomorphism established in Proposition \ref{prop:adm_1_der}, we get a natural isomorphism of modules:
\[
\operatorname{H}^{0}(\Hom_{B}(\nu_{B}(L)\otimes_{A} M,N))=\operatorname{H}^{0}(\Hom_{A}\big(M, \nu_{A}\big(\Hom_{B}(\nu_{B}(L),N)\big)\big)).
\]
This establishes the corollary. 
\end{proof}

\begin{lemma} \label{lem:502}
Let $L\in \overline{\cat{C}}(A\otimes B^{\mathrm{op}})$ and $N\in \cat{C}(B)$. Then, there is a natural isomorphism in $\cat{C}(A)$:
\[ \nu_{A}(L)\otimes_{B} N\cong \nu_{A}(L \otimes_{B} N). \]
\end{lemma}
\begin{proof}
The proof of this lemma  follows from Lemma \ref{lem:pre_4} and the fact that $\nu_{A}$ commutes with arbitrary direct sums (Lemma \ref{lem:101}).
\end{proof}

\begin{lemma}\label{lem:102}
Let $L\in \overline{\cat{D}}(A\otimes B^{\mathrm{op}})$ and $N\in \cat{D}(B)$. Then, there is a natural isomorphism in $\cat{D}(A)$:
\[ \nu_{A}(L)\otimes_{B}^{\mathrm{L}} N\cong \nu_{A}(L \otimes_{B}^{\mathrm{L}} N). \]
\end{lemma}
\begin{proof}
Given $N\in \cat{D}(B)$ choose a K-projective resolution $P\to N$. Then, $L\otimes_{B}^{\mathrm{L}} N \cong L\otimes_{B} P$, and $\nu_{A}(L)\otimes_{B}^{\mathrm{L}} N\cong \nu_{A}(L)\otimes_{B} P$. Now use Lemma \ref{lem:502}.
\end{proof}

\begin{lemma} \label{lem:703}
Let $M\in \overline{\cat{K}}(A)$. Then, the following are equivalent:
\begin{enumerate}
\rmitem{a.} $M\in \cat{K}(A)$.
\rmitem{b.} $A\otimes_{A} M \cong M$.
\rmitem{c.} $\nu_{A}(\mathrm{Hom}_{A}(A,M))\cong M$. 
\end{enumerate}
\end{lemma}
\begin{proof}
This is straightforward, using Lemma \ref{lem:pre_3}.
\end{proof}

\begin{lemma}\label{lem:103}
Let $M\in \overline{\cat{D}}(A)$. Then, the following are equivalent:
\begin{enumerate}
\rmitem{a.} $M\in \cat{D}(A)$.
\rmitem{b.} $A\otimes_{A}^{\mathrm{L}} M \cong M$.
\rmitem{c.} $\nu_{A}(\RHom_{A}(A,M))\cong M$. 
\end{enumerate}
\end{lemma}
\begin{proof}
Suppose $M\in \cat{D}(A)$. Let $P\to M$ be a K-projective resolution of $M$ in $\cat{K}(A)$. 

We wish to show (b.). Using Lemma \ref{lem:703}, we conclude that $A\otimes_{A} P \cong P$. However, this precisely says that $A\otimes_{A}^{\mathrm{L}} M\cong M$. 

We now come to (c.). Let $M\to I$ be a K-injective resolution of $M$ in $\overline{\cat{K}}(A)$; by Lemma \ref{cor:201}, $\nu_{A}(I)$ is a K-injective resolution of $M$ in $\cat{K}(A)$. Thus we have an isomorphism $\nu_{A}(\RHom_{A}(A,M))\cong \nu_{A}(\Hom_{A}(A,\nu_{A}(I)))$. We now use Lemma \ref{lem:703} to conclude that  $\nu_{A}(\Hom_{A}(A,\nu_{A}(I)))\cong \nu_{A}(I)\cong M$. This proves (c.).

If either (b.) or (c.) hold, it is obvious that $M\in \cat{D}(A)$.
\end{proof}

\begin{prop}\label{prop:adm_1_der2}
Let $M\in \cat{D}(A)$, $N\in \cat{D}(B)$ and $L\in \overline{\cat{D}}(B\otimes A^{{\rm op}})$. Then, there is a natural isomorphism in $\cat{D}(\mathbb{K})$:
\[
\RHom_{B}(\nu_{B}(L)\otimes_{A}^{{\rm L}} M,N)\cong \RHom_{A}\big(M, \nu_{A}\big(\RHom_{B}(\nu_{B}(L),N)\big)\big).
\]
\end{prop}
\begin{proof}
Let $P\to M$ be a K-projective resolution of $M$ in $\cat{K}(A)$. Similarly, consider a K-injective resolution $N\to I$ of $N$ in $\cat{K}(B)$. Applying Proposition \ref{prop:adm_1_der} to $P$, $I$, and $L$, we have an isomorphism of complexes \[
\Hom_{B}(\nu_{B}(L)\otimes_{A}P,I)\cong \Hom_{A}\big(P, \nu_{A}\big(\Hom_{B}(\nu_{B}(L),I)\big)\big).
\]
This precisely says that \[
\RHom_{B}(\nu_{B}(L)\otimes_{A}^{{\rm L}} M,N)\cong \RHom_{A}\big(M, \nu_{A}\big(\RHom_{B}(\nu_{B}(L),N)\big)\big). \qedhere\]
\end{proof}

\begin{corollary} \label{cor:81}
Let $M\in \cat{D}(A)$, $N\in \cat{D}(B)$ and $L\in \overline{\cat{D}}(B\otimes A^{{\rm op}})$. Then, there is a natural isomorphism of modules:
\[
\Hom_{\cat{D}(B)}(\nu_{B}(L)\otimes_{A}^{{\rm L}} M,N)\cong \Hom_{\cat{D}(A)}\big(M, \nu_{A}\big(\RHom_{B}(\nu_{B}(L),N)\big)\big).
\]
\end{corollary}
\begin{proof}
For complexes $M'$ and $N'$ in $\cat{D}(A)$, there is an isomorphism $\Hom_{\cat{D}(A)}(M',N')$  $\cong \operatorname{H}^{0}(\RHom_{A}(M',N'))$. The corollary now follows from Proposition \ref{prop:adm_1_der2}.
\end{proof}

We now obtain the following adjunctions from Corollary \ref{cor:81}.
\begin{corollary}\label{corr:adm_1_der}
Let $\phi:A\to B$ be a morphism of algebras. Let $M\in \cat{D}(A)$ and $N\in \cat{D}(B)$. Then, there are the following natural identities:
\begin{enumerate}
\rmitem{i} $\Hom_{\cat{D}(A)}(\nu_{A}(N),M)\cong\Hom_{\cat{D}(B)}(N,\nu_{B}(\RHom_{A}(\nu_{A}(B),M)))$.
\rmitem{ii} $\Hom_{\cat{D}(B)}(B\otimes_{A}^{{\rm L}} M,N)\cong\Hom_{\cat{D}(A)}(M,\nu_{A}(N))$.
\end{enumerate}
\end{corollary}
\begin{proof}
We first prove (ii). Consider $B$ as an object in $\overline{\cat{D}}(B\otimes A^{{\rm op}}).$ From Corollary \ref{cor:81}, we have an isomorphism 
\[ \Hom_{\cat{D}(B)}(\nu_{B}(B)\otimes_{A}^{{\rm L}} M,N)\cong \Hom_{\cat{D}(A)}\big(M, \nu_{A}\big(\RHom_{B}(\nu_{B}(B),N)\big)\big).\] By Lemma \ref{lem:102} and Lemma \ref{lem:103}, there is an isomorphism $\nu_{B}(B)\otimes_{A}^{\mathrm{L}} M \cong B\otimes_{A}^{\mathrm{L}} M.$ Similarly, $\nu_{B}(\RHom_{B}(\nu_{B}(B),N)) \cong N$. Thus (ii) follows. 

For the proof of (i), we need to invert the roles of $A$ and $B$, and consider $B$ as an object in $\overline{\cat{D}}(B\otimes A^{{\rm op}}).$ Thus, we have an isomorphism 
\[ \Hom_{\cat{D}(A)}(\nu_{A}(B)\otimes_{B}^{{\rm L}} N,M)\cong \Hom_{\cat{D}(B)}\big(N, \nu_{B}\big(\RHom_{A}(\nu_{A}(B),M)\big)\big).\] To prove (i), all we need to show now is that $\nu_{A}(B)\otimes_{B}^{\mathrm{L}} N \cong \nu_{A}(N)$; this follows immediately from Lemma \ref{lem:102} and Lemma \ref{lem:103}.
\end{proof}

\section{Special idempotents, split idempotents and categorical decompositions}\label{s:spc_idm}\
Before moving on to the results in this section, we survey a concrete example coming from the theory of complex representations of $p$-adic groups. This example is purely for context; a reader not familiar with the subject can skip it entirely.   
\subsection{The case of reductive groups over non-archimedean local fields}\label{ss:red_grp}
Let $F$ be a non-archimedean local field and let $G$ denote the topological group comprising of $F$-rational points of a connected reductive algebraic group defined over it. Consider the collection of all pairs $(M,\sigma)$ where $M$ is an $F$-Levi subgroup of $G$ and $\sigma$ is a supercuspidal representation of $M$. Define an equivalence relation on the collection of such pairs by deeming two such pairs $(M_{1},\sigma_{1})$ and $(M_{2},\sigma_{2})$ to be equivalent if there exists an element $g\in G$ and unramified quasi-character $\chi$ of $M_{2}$ such that
\[
M_{2}=g^{-1}M_{1}g\ {\rm and} \ \sigma_{2}=\sigma_{1}^{g}\otimes \chi,
\]
where $\sigma_{1}^{g}$ is the representation of $M_{2}$ given by composition of $\sigma_{1}$ by the $g$-conjugation map. The equivalence classes of this relation are known as {\it inertial supports}, and the set of all inertial supports for $G$ is denoted by $\mathfrak B(G)$.

Let $\mathfrak R(G)$ denote the category of smooth complex representations of $G$. Bernstein provided a direct product decomposition of this abelian category (see \cite{Be}, \cite{Ber}) into full subcategories indexed by inertial supports, that is, 
\[
\mathfrak R(G)=\prod_{\mathfrak s\in \mathfrak B(G)}\mathfrak R^{\mathfrak s}(G).
\]

Let $\mathcal H(G)$ denote the Hecke algebra of $G$ consisting of all locally constant, compactly supported, complex valued functions on $G$. It is well known that the category of non-degenerate left modules over $\mathcal H(G)$ is equivalent to $\mathfrak R(G)$. For an idempotent $e\in \mathcal H(G)$, set $\mathfrak R_{e}(G)$ to be the full sub-category of $\mathfrak R(G)$ whose objects are those $V\in \mathfrak R(G)$ which satisfy the condition that  $V=\mathcal H(G) e V,$ i.e. $V$ is generated as a $G$-module by its subspace of $e$ fixed vectors. It is shown in \cite[Proposition 3.3]{BK} that the category $\mathfrak R_{e}(G)$ is closed relative to its $G$-subquotients if and only if it is naturally equivalent to the category of left modules over the unital ring $e\mathcal H(G)e$. The idempotents $e\in \mathcal H(G)$ which satisfy this aforementioned property (i.e. the category $\mathfrak R_{e}(G)$ is closed relative to its $G$-subquotients) are known as {\it special} idempotents. An example of a special idempotent is the characteristic function of an Iwahori spherical subgroup of $G$ (see \cite[Lemma 4.8]{B},\cite{Mat}; also c.f. \cite[\S 9.2]{BK}). In this case the algebra $e\mathcal H(G)e$ is the well known Iwahori Hecke algebra.

For any finite subset $\mathfrak S$ of $\mathfrak B(G)$ define $\mathfrak R^{\mathfrak S}(G)=\prod_{\mathfrak s\in \mathfrak S}\mathfrak R^{\mathfrak s}(G)$. It is shown in \cite[Proposition 3.6]{BK} that given any special idempotent $e\in \mathcal H(G)$, there exists a finite subset $\mathfrak S$ of $\mathfrak B(G)$ such that $\mathfrak R_{e}(G)=\mathfrak R^{\mathfrak S}(G)$. Conversely, given a finite subset $\mathfrak S$ of $\mathfrak B(G)$, there exists a special idempotent $e(\mathfrak S)\in \mathcal H(G)$ such that $\mathfrak R^{\mathfrak S}(G)=\mathfrak R_{e(\mathfrak S)}(G)$ (see \cite[Proposition 3.13]{BK}). These results on special idempotents are the basis of the ``theory of types" in the representation theory of $p$-adic groups.  

Recently in \cite[\S 1]{SZ}, some of the above results on special idempotents have been proved for the module category over $\mathcal S(G)$, where $\mathcal S(G)$ denotes the Harish-Chandra Schwartz algebra of $G$.

\subsection{Special idempotents in module category} \label{sim}
 We now extend some notions mentioned in \S \ref{ss:red_grp} to module categories over arbitrary idempotented algebras. 
\begin{lemma} \label{lem:10}
Let $e\in \Ai$. The following two subcategories of $\cat{M}(A)$ coincide:
\begin{enumerate}
 \rmitem{a.} $\{M\in \cat{M}(A)\ \vert \ \forall\  K\lneq H\leq M, e(H/K)\neq 0\}$,
 \rmitem{b.} $\{M\in \cat{M}(A)\ \vert \ \forall\  K\lneq H\leq M, H/K \ \text{simple},\ e(H/K)\neq 0\}$.
\end{enumerate}
\end{lemma}
\begin{proof}

Suppose $M\in \cat{M}(A)$. If every subquotient of $M$ admits an $e$-fixed element, then every irreducible subquotient does so as well. Conversely, suppose every irreducible subquotient of $M$ admits an $e$-fixed element, and let $N$ be a non-zero subquotient of $M$. Taking quotients, we can assume that $N\leq M$. 

Pick $n\in N$, and consider the submodule $An$; note that $n\in An$. By Zorn's lemma $An$ admits an irreducible quotient: by hypothesis, this irreducible quotient must admit an $e$-fixed element. Since $Ae$ is projective, this element lifts to an $e$-fixed element in $An$, and thus in $N$. The result follows.
\end{proof}

\begin{definition} \label{dfn:10}
Denote the subcategory of $\cat{M}(A)$ under consideration in Lemma \ref{lem:10} by $\cat{M}(A)_{e}$.
\end{definition}

\begin{remark}\label{rem_lem:10b}
It is clear from the characterization in Lemma \ref{lem:10}(b.) that $\cat{M}(A)_{e}$ is closed under subobjects, quotients, and extensions. As such, it is an abelian category in its own right, and the canonical inclusion functor is exact. Moreover, $\cat{M}(A)_{e}$ is also closed under arbitrary direct sums.
\end{remark}

We now extend \cite[Definition 3.11]{BK} to general idempotented algebras.
 \begin{definition} \label{dfn:02}
 An idempotent $e\in \Ai$ is said to be \emph{left special} if the full subcategory $\{M\in \cat{M}(A)\ \vert \ M = AeM\}$ is closed under subquotients. 
 \end{definition}
 
 \begin{remark}
 It is clear that the full subcategory  $\{M\in \cat{M}(A)\ \vert \ M = AeM\}$ is always closed under quotients. Thus, $e\in \Ai$ is left special if and only if  $\{M\in \cat{M}(A)\ \vert \ M = AeM\}$ is closed under subobjects.
 \end{remark}

 \begin{lemma} \label{lem:02}
 Let $e\in \Ai$ be left special. Then $\cat{M}(A)_{e} = \{M\in \cat{M}(A)\ \vert \ M = AeM\}$.
\end{lemma}
\begin{proof}
Since $e$ is left special the category $\{M\in \cat{M}(A)\ \vert \ M = AeM\}$ is closed under subquotients. Thus, if $M = AeM$, then every subquotient of $M$ is also generated by its $e$-fixed elements; in particular, every non-zero subquotient must admit an $e$-fixed element. Conversely, suppose every non-zero subquotient of $M$ admits an $e$-fixed element. If $M\neq AeM$, the quotient $M/AeM$ must do so as well. By the projectivity of $Ae$, such an element must lift to an $e$-fixed element of $M$ lying outside $AeM$. This is a contradiction.
\end{proof}

A full subcategory $\cat{D}$ of a category $\cat{C}$ is called \emph{localizing} if the canonical inclusion functor $\cat{C}\to \cat{D}$ admits a right adjoint. 

\begin{lemma}\label{lem:03}
 Let $e\in \Ai$ be left  special. The category $\cat{M}(A)_{e}$ is a localizing subcategory of $\cat{M}(A)$. The inclusion \[\operatorname{inc} : \cat{M}(A)_{e}\to \cat{M}(A)\] admits a right adjoint \[\Gamma_{e}: \cat{M}(A)\to \cat{M}(A)_{e},\] where $\Gamma_{e}(M) = AeM$ for $M\in \cat{M}(A)$.
 \end{lemma}
 \begin{proof}
 Let $M\in \cat{M}(A)_{e}$ and $N\in \cat{M}(A)$, and suppose $f:M\to N$ is a map of $A$-modules. Since $M = AeM$, $f(M) = Aef(M)\subseteq AeN$. It follows that every map from $M$ to $N$ factors through $AeN = \Gamma_{e}(N)$.  
 \end{proof}
 
 \begin{remark}
 In a Grothendieck abelian category, every full subcategory closed under extensions, quotients, subobjects, and arbitrary coproducts is localizing. Thus, even without the assumption that $e$ is left special, it is still the case that $\cat{M}(A)_{e}$ is localizing (see Remark \ref{rem_lem:10b}). When $e$ is not known to be left special, however, we have no explicit description of the functor $\Gamma_{e}$. 
 \end{remark}
 
  \begin{lemma} \label{lem:05} 
 Let $e\in \Ai$ be left special. Then, $Ae$ is a finitely generated projective generator for the category $\cat{M}(A)_{e}$. 
 \end{lemma}
 \begin{proof}
 This is a consequence of Lemma \ref{lem:04} and Lemma \ref{lem:02}.
 \end{proof}

 \begin{theorem} [c.f.~Proposition 3.3 of \cite{BK}] \label{prop:02}
 Let $e\in \Ai$ be left special. The functor \[\mathrm{Hom}_{A}(Ae, -): \cat{M}(A)_{e}\to \cat{M}(eAe)\] is an equivalence of categories, with quasi-inverse \[A\otimes_{eAe} (-): \cat{M}(eAe)\to \cat{M}(A)_{e}.\]
 \end{theorem}
 \begin{proof}
 By Lemma \ref{lem:05}, $Ae$ is a finitely generated projective generator for the category $\cat{M}(A)_{e}$. The category $\cat{M}(A)_{e}$ is an abelian category admitting arbitrary coproducts. By Morita theory (\cite[II Theorem 1.3]{Bas68}) the functor $\mathrm{Hom}_{A}(Ae, -)$ establishes an equivalence of categories between $\cat{M}(A)_{e}$ and $\cat{M}(\mathrm{End}_{A}(Ae)^{\mathrm{op}})$. By Lemma \ref{lem:04}, $\mathrm{End}_{A}(Ae)^{\mathrm{op}}\cong eAe$. Let us denote a quasi-inverse to $\mathrm{Hom}_{A}(Ae, -)$ by $\operatorname{H}$.
 
 We can construct the following diagram of functors:
\[  
\begin{tikzcd}[row sep=huge, column sep = huge]
{} & \cat{M}(A) \arrow[yshift = -.7ex]{ld}{\Gamma_{e}} \arrow[yshift = -.7ex]{rd}[swap]{\mathrm{Hom}_{A}(Ae, -)} &  {}  \\
\cat{M}(A)_{e} \arrow[yshift = .7ex]{ur}{\mathrm{inc}}  \arrow[yshift = .7ex]{rr}{\mathrm{Hom}_{A}(Ae, -)} & {}  & \cat{M}(eAe)  \arrow[yshift = .7ex]{lu}[swap]{A\otimes_{eAe} (-)} \arrow[yshift = -.7ex]{ll}{\operatorname{H}}
 \end{tikzcd}
 \]

By construction the functor $\operatorname{inc}$ is left adjoint to $\Gamma_{e}$. Corollary \ref{corr:adm_1} implies that as functors between $\cat{M}(A)$ and $\cat{M}(eAe)$, the functor $A\otimes_{eAe} (-)$ has $\nu_{eAe}$ as its right adjoint. As a functor from $\cat{M}(A)$ to $\cat{M}(eAe)$, however, $\nu_{eAe}$ is nothing but $\mathrm{Hom}_{A}(Ae, -)$. Finally, we can choose $\operatorname{H}$ to be left adjoint to $\mathrm{Hom}_{A}(Ae,-)$.
 
 Thus, on each of the three sides of the above diagram, the functor represented by the inner arrow is right adjoint to that represented by the outer arrow. Furthermore, the triangular diagram formed by considering only the inner arrows is commutative. 
 
As adjoints compose, $\operatorname{inc}\circ \operatorname{H} \cong A\otimes_{eAe} (-)$. This establishes the theorem. 
\end{proof}
 
\begin{corollary}
Let $e\in \Ai$ be left special. Then, $Ae$ is flat as a right $eAe$-module. 
\end{corollary}
\begin{proof}
By Theorem \ref{prop:02}, the functor $A\otimes_{eAe} (-)$ establishes an equivalence of categories between $\cat{M}(eAe)$ and $\cat{M}(A)_{e}$; as such, it is exact. As $\cat{M}(A)_{e}$ is an abelian subcategory of $\cat{M}(A)$, this precisely says that any exact sequence in $\cat{M}(eAe)$ remains exact after applying the functor $A\ot _{eAe}(-)$. However, for any $M\in \cat{M}(eAe)$ there is a natural isomorphism $A\ot _{eAe} M \cong Ae\ot_{eAe} M$: thus $Ae$ is flat as a right $eAe$-module.
\end{proof}

\subsection{Split idempotents and the derived category} \label{sid}

\begin{definition}\label{defn:split}
An idempotent  $e\in \Ai$ is said to be \emph{left split} if the inclusion $\operatorname{inc}: \cat{M}(A)_{e}\to \cat{M}(A)$ preserves injective objects. 
\end{definition}

For example, if $e$ is a special idempotent arising from a finite subset of $\mathfrak{B}(G)$ (see \S \ref{ss:red_grp}), then $e$ is left split: this follows from Bernstein's direct product decomposition of the category of smooth complex representations of $G$.

\begin{lemma} \label{lem:07}
An idempotent $e\in \Ai$ is left split if and only if for all injective objects $I\in \cat{M}(A)$, $\Gamma_{e}(I)$\footnote{The functor $\Gamma_{e}$ here is actually the functor $\operatorname{inc}\circ\Gamma_{e}$.} is injective in $\cat{M}(A)$.
\end{lemma}
\begin{proof}
Suppose first that $e$ is left split, and $I$ is an injective object in $\cat{M}(A)$. The functor $\Gamma_{e}$ is right-adjoint to the inclusion $\operatorname{inc}$; as $\operatorname{inc}$ is exact, $\Gamma_{e}$ sends injective objects in $\cat{M}(A)$ to injective objects in $\cat{M}(A)_{e}$. As $e$ is left split, $\Gamma_{e}(I)$ is therefore injective in $\cat{M}(A)$.

Conversely assume that $\Gamma_{e}$, viewed as a functor from $\cat{M}(A)$ to itself, preserves injective objects. Suppose $I$ is injective in $\cat{M}(A)_{e}$, and let $E$ be its injective hull in $\cat{M}(A)$. By hypothesis, $\Gamma_{e}(E)$ is injective in $\cat{M}(A)$; it is also an object of $\cat{M}(A)_{e}$. Since $I\leq \Gamma_{e}(E)$, $I$ is therefore a summand of $\Gamma_{e}(E)$. As $\Gamma_{e}(E)$ is injective in $\cat{M}(A)$, $I$ is so as well. 
\end{proof}

\begin{lemma} \label{lem:11}
Let $e\in \Ai$ be left special, and suppose $N\in \cat{M}(A)$. Then, $\mathrm{Hom}_{A}(M,N)=0$ for all $M\in \cat{M}(A)_{e}$ if and only if $eN=0$.
\end{lemma}
\begin{proof}
If $\mathrm{Hom}_{A}(M,N)=0$ for all $M\in \cat{M}(A)_{e}$, then  $\Gamma_{e}(N) = AeN = 0$. Thus $eN=0$. 

Conversely, if $f:M\to N$ is a non-zero map where $M \in \cat{M}(A)_{e}$, then $0\neq \operatorname{im}(f)\leq AeN$. If $eN=0$, this is a contradiction.
\end{proof}

\begin{lemma} \label{lem:12}
Let $e\in \Ai$ be left special. Suppose $M\in \cat{M}(A)_{e}$ and $eN=0$. Then $\mathrm{Hom}_{A}(N,M)=0$.
\end{lemma}
\begin{proof}
The existence of a non-zero morphism $f:N\to M$ implies the existence of a non-zero submodule $0\neq f(N)\leq M$. However, by Lemma \ref{lem:02} and Lemma \ref{lem:10}, we know that $f(N)$ must contain a non-zero $e$-fixed element: thus, $ef(N)\neq 0$. This is a contradiction. 
\end{proof}

 \begin{proposition} \label{prop:20}
 Let $e\in \Ai$ be left special and left split. Suppose $M\in \cat{M}(A)_{e}$ and $eN=0$. Then $\mathrm{RHom}_{A}(N,M) = 0$. 
 \end{proposition}
 \begin{proof}
 Let $0\to M\to I^{0}\to I^{1}\to I^{2}\to \ldots$ be an injective resolution of $M$ in $\cat{M}(A)_{e}$. Since $e$ is left split, $I^{0}\to I^{1}\to I^{2}\to \ldots$ remains an injective resolution of $M$ in $\cat{M}(A)$, and can be used to compute $\mathrm{RHom}_{A}$.
As each $I^{i}$ is in $\cat{M}(A)_{e}$, Lemma \ref{lem:12} implies that $\mathrm{Hom}_{A}(N, I^{i}) = 0$ for all $i\in \mathbb{N}$. The conclusion follows.
 \end{proof}
 
 \begin{corollary} \label{cor:10}
  Let $e\in \Ai$ be left special and left split. Suppose $M\in \cat{M}(A)_{e}$ and $eN=0$. Then $\mathrm{Ext}^{1}_{A}(N,M) = 0$. 
  \end{corollary}
  \begin{proof}
This follows from the isomorphism $\mathrm{Ext}^{1}_{A}(N,M) \cong \operatorname{H}^{1}(\mathrm{RHom}_{A}(N,M))$.
  \end{proof}

\begin{corollary} \label{cor:12}
Let $e\in \Ai$ be left special and left split, and suppose $M\in \cat{M}(A)$. Then, the short exact sequence $0\to \Gamma_{e}(M)\to M\to M/\Gamma_{e}(M)\to 0$ splits. In particular,
$
M\cong \Gamma_{e}(M)\oplus M/\Gamma_{e}(M).
$
\end{corollary}
\begin{proof}
The short exact sequence $0\to \Gamma_{e}(M)\to M\to M/\Gamma_{e}(M)\to 0$ represents an element of $\mathrm{Ext}_{A}^{1}(M/\Gamma_{e}(M), \Gamma_{e}(M))$. By Corollary \ref{cor:10}, $\mathrm{Ext}_{A}^{1}(M/\Gamma_{e}(M), \Gamma_{e}(M)) = 0$. Thus, the short exact sequence is split.
\end{proof}

\begin{definition} \label{dfn:30}
Let $e\in \Ai$ be left special. Denote the full subcategory $\{M\in \cat{M}(A)\ |\  eM=0\}$ by $\cat{M}(A)_{e}^{\perp}$. 
\end{definition}

 If $\cat{E}$ and $\cat{F}$ are arbitrary categories, we can form the product category $\cat{E}\times \cat{F}$. There are canonical projection functors $\cat{E}\times \cat{F} \to \cat{E}$ and $\cat{E} \times \cat{F} \to \cat{F}$. When $\cat{E}$ and $\cat{F}$ are additive categories there are also natural inclusion functors $\cat{E}\to \cat{E}\times \cat{F}$ and $\cat{F}\to \cat{E}\times \cat{F}$. For example the inclusion functor $\cat{E}\to \cat{E}\times \cat{F}$ sends $X\in \mathrm{Ob}(\cat{E})$ to $(0,X)$. In the case where $\cat{E}$ and $\cat{F}$ are additive, we will denote $\cat{E}\times \cat{F}$ by $\cat{E} \bigoplus \cat{F}$; $\cat{E} \bigoplus \cat{F}$ is again additive. When $\cat{E}$ and $\cat{F}$ are abelian, $\cat{E}\bigoplus \cat{F}$ is so as well.

\begin{proposition} \label{prop:30}
 Let $e\in \Ai$ be left special and left split.  Then, $\cat{M}(A) = \cat{M}(A)_{e} \bigoplus \cat{M}(A)_{e}^{\perp}$.
 \end{proposition}
 \begin{proof}
 This follows from Lemma \ref{lem:11}, Lemma \ref{lem:12}, and Corollary \ref{cor:12}.
 \end{proof}
 
 \begin{corollary} \label{cor:41}
 Let $e\in \Ai$ be left special and left split. Then, the functor $\Gamma_{e}$ is exact.
\qed
\end{corollary}

 As $\Gamma_{e}$ is exact, it extends to a functor $\cat{D}(A)\to \cat{D}(A)$; we continue to denote this extension by $\Gamma_{e}$.
 We now extend Proposition \ref{prop:30} to the derived category. 
 
 \begin{lemma} \label{l:078}
 Let $e\in \Ai$ be left special and left split. 
 \begin{enumerate}
  \rmitem{i} Let $M\in \cat{K}(A)$ be K-injective. Then, $\Gamma_{e}(M)$ is K-injective.
  \rmitem{ii} Let $M\in \cat{K}(A)$ be K-projective. Then, $\Gamma_{e}(M)$ is K-projective.
 \end{enumerate}
 \end{lemma} 
 \begin{proof}
 Let $M\in \cat{K}(A)$. By Proposition \ref{prop:30}, $M$ decomposes as a direct sum $M = \Gamma_{e}(M)\oplus \mathrm{Q}_{e}(M)$ in $\cat{C}(A)$, where $\mathrm{Q}_{e}(M)^{i}$ is the largest $A$-submodule of $M^{i}$ annihilated by $e$. Direct summands of K-injective and K-projective complexes are, respectively, K-injective and K-projective; the lemma follows.
 \end{proof}

\begin{definition} \label{dfn:50}
Denote the full triangulated subcategory $\{M\in \cat{D}(A)\ |\  \Gamma_{e}(M)\cong M\}$ by $\cat{D}(A)_{e}$.
\end{definition}

\begin{lemma} \label{l:079}
Let $e\in \Ai$ be left special and left split. Then, $\cat{D}(A)_{e} = \{M\in \cat{D}(A)\ \vert \ \operatorname{H}^{i}(M)\in \cat{M}(A)_{e},\  \forall\  i\in \mathbb{Z}\}.$
\end{lemma}
\begin{proof}
First, suppose that $M$ lies in the class $\{M\in \cat{D}(A)\ \vert \ \operatorname{H}^{i}(M)\in \cat{M}(A)_{e},\  \forall\  i\in \mathbb{Z}\}.$ Since $\Gamma_{e}$ is exact, this implies that the canonical map $\Gamma_{e}(M)\to M$ is an isomorphism in $\cat{D}(A)$. 
Conversely, if $M\in \cat{D}(A)_{e}$, the canonical map $\Gamma_{e}(M)\to M$ is an isomorphism. By taking cohomology and again using the exactness of $\Gamma_{e}$, we see that $M$ belongs to  $\{M\in \cat{D}(A)\ \vert \ \operatorname{H}^{i}(M)\in \cat{M}(A)_{e},\  \forall\  i\in \mathbb{Z}\}.$
\end{proof}

\begin{definition} \label{d:079}
Let $e\in \Ai$ be left special and left split. Denote the full triangulated subcategory $\{M\in \cat{D}(A)\ |\ \operatorname{H}^{i}(M) \in \cat{M}(A)_{e}^{\perp},\ \forall\  i\in \mathbb{Z}\}$ by  $\cat{D}(A)_{e}^{\perp}$.
\end{definition}

\begin{lemma} \label{l:082}
Let $e\in \Ai$ be left special and left split.  Then, $\cat{D}(A)_{e}^{\perp} = \{M\in \cat{D}(A)\ |\ \mathrm{RHom}_{A}(Ae,M)= 0\} = \{M\in \cat{D}(A)\ |\ eA\otimes_{A}^{\mathrm{L}} M= 0\}.$
\end{lemma}
\begin{proof}
Since $Ae$ and $eA$ are projective objects in $\cat{M}(A)$ and $\cat{M}(A^{\mathrm{op}})$, respectively, for all $i\in \mathbb{Z}$ there are isomorphisms \[\operatorname{H}^{i}(\mathrm{RHom}_{A}(Ae,M)) \cong \operatorname{H}^{i}(eA\otimes_{A}^{\mathrm{L}} M) \cong e\operatorname{H}^{i}(M).\] The lemma follows. 
\end{proof}

\begin{lemma} \label{l:081}
Let $e\in \Ai$ be left special and left split, and suppose $M\in \cat{D}(A)_{e}$ and $N\in \cat{D}(A)_{e}^{\perp}.$ Then, \[\mathrm{RHom}_{A}(M,N) = \mathrm{RHom}_{A}(N,M) = 0.\]
\end{lemma}
\begin{proof}

As in the proof of Lemma \ref{l:078}, $N$ decomposes as a direct sum $N = \Gamma_{e}(N)\oplus \mathrm{Q}_{e}(N)$, where $\mathrm{Q}_{e}(N)$ is a complex each of whose terms is annihilated by $e$. Since $N\in \cat{D}(A)_{e}^{\perp}$, $\Gamma_{e}(N) = 0$ in $\cat{D}(A)$. Thus, there is no loss of generality in assuming that each term of $N$ is annihilated by $e$.

Let $P\to M$ and $M\to I$ be K-projective and K-injective resolutions of $M$, respectively. Since $M\in \cat{D}(A)_{e}$, Lemma \ref{l:082} implies that $M$ is actually isomorphic to both $\Gamma_{e}(P)$ and $\Gamma_{e}(I)$ in $\cat{D}(A)$; by Lemma \ref{l:078}, $\Gamma_{e}(P)$ and $\Gamma_{e}(I)$ are still, respectively, K-projective and K-injective. It follows that $\mathrm{RHom}_{A}(M,N) = \mathrm{Hom}_{A}(\Gamma_{e}(P),N)$, and $\mathrm{RHom}_{A}(N,M) = \mathrm{Hom}_{A}(N,\Gamma_{e}(I))$.

The result now follows from Lemma \ref{lem:11} and Lemma \ref{lem:12}.
\end{proof}

\begin{proposition} \label{cor:21}
 Let $e\in \Ai$ be left special and left split.  Then, $\cat{D}(A) = \cat{D}(A)_{e} \bigoplus \cat{D}(A)_{e}^{\perp}.$
\end{proposition}
\begin{proof}
This follows from Lemma \ref{l:079} and Lemma \ref{l:081}.
\end{proof}

\begin{remark}
Proposition \ref{cor:21} provides a derived analogue of Bernstein's decomposition of categories when the algebra $A$ is the Hecke algebra of a connected reductive algebraic group defined over a non-archimedean local field.
\end{remark}

\begin{theorem} \label{prop:51}
 Let $e\in \Ai$ be left special and left split. There is a commutative diagram of functors:
 \[  
\begin{tikzcd}[row sep=huge, column sep = huge]
{} & \cat{D}(A) \arrow[yshift = -.7ex]{ld}{\Gamma_{e}} \arrow[yshift = -.7ex]{rd}[swap]{\mathrm{RHom}_{A}(Ae, -)} &  {}  \\
\cat{D}(A)_{e} \arrow[yshift = .7ex]{ur}{\mathrm{inc}}  \arrow[yshift = .7ex]{rr}{\mathrm{RHom}_{A}(Ae, -)} & {}  & \cat{D}(eAe)  \arrow[yshift = .7ex]{lu}[swap]{A\otimes^{\mathrm{L}}_{eAe} (-)} \arrow[yshift = -.7ex]{ll}{A\otimes^{\mathrm{L}}_{eAe} (-)}
 \end{tikzcd}
 \]
 The functor \[\mathrm{RHom}_{A}(Ae, -): \cat{D}(A)_{e}\to \cat{D}(eAe)\] is an equivalence of categories, with quasi-inverse \[A\otimes^{\mathrm{L}}_{eAe} (-): \cat{D}(eAe)\to \cat{D}(A)_{e}.\]
\end{theorem}
\begin{proof}
The proof of this result is similar to  that of Theorem \ref{prop:02}. Since the functor $\Gamma_{e}: \cat{D}(A)\to \cat{D}(A)$ is idempotent, the functor $\operatorname{inc}$ is left adjoint to $\Gamma_{e}$. By Corollary \ref{corr:adm_1_der}, we know that the functor  $A\otimes^{\mathrm{L}}_{eAe} (-)$ is left adjoint to $\mathrm{RHom}_{A}(Ae, -)$ as functors between $\cat{D}(A)$ and $\cat{D}(eAe)$.

We are left with considering the functors on the lower edge of the diagram. If $M \in \cat{D}(A)_{e}$ is such that $\mathrm{RHom}_{A}(Ae,M) = 0$, then $M = 0$ by Lemma \ref{l:082} and Proposition \ref{cor:21}. From this, we see that $Ae$ is a generator for the category $\cat{D}(A)_{e}$; it is also compact. Since the triangulated category $\cat{D}(A)_{e}$ is algebraic (\cite[Lemma 7.5]{Kr}), by \cite[Theorem 3.3]{Ke} the functor $\mathrm{RHom}_{A}(Ae,-)$ establishes an equivalence of categories between $\cat{D}(A)_{e}$ and $\cat{D}(\mathrm{RHom}_{A}(Ae,Ae))^{\mathrm{op}}$. Since $Ae$ is projective, $\mathrm{RHom}_{A}(Ae,Ae)$ is nothing by $eAe^{\mathrm{op}}$ (Lemma \ref{lem:04}). The fact that $A\otimes^{\mathrm{L}}_{eAe} (-)$ is quasi-inverse to $\mathrm{RHom}_{A}(Ae, -)$ can be argued as in Theorem \ref{prop:02}.
\end{proof}

\subsection{Decompositions}\label{ss:decom}

\begin{lemma} \label{lem:891}
Let $e,e' \in \Ai$ be left special. Then, $eAe' = 0$ if and only if $\cat{M}(A)_{e}\cap \cat{M}(A)_{e'} = 0$.
\end{lemma}
\begin{proof}
Suppose $eAe'=0$, and let $N\in \cat{M}(A)_{e}\cap \cat{M}(A)_{e'}$. Then, by definition, $N = AeAe'N$. Thus $N=0$. Conversely, suppose $\cat{M}(A)_{e}\cap \cat{M}(A)_{e'} = 0$. If $f\in \mathrm{Hom}_{A}(Ae,Ae')$, $f(Ae)\in \cat{M}(A)_{e}\cap \cat{M}(A)_{e'}$. Thus $\mathrm{Hom}_{A}(Ae,Ae')=0$. By Lemma \ref{lem:pre_5}, $eAe' \cong \mathrm{Hom}_{A}(Ae,Ae')$. This completes the argument.
\end{proof}

\begin{corollary} \label{cor:892}
Let $e,e' \in \Ai$ be left special. Then, $eAe' = 0$ if and only if $e'Ae = 0$.
\qed
\end{corollary}

\begin{definition} \label{def:893}
Let $e,e' \in \Ai$ be left special. Then $e$ and $e'$ are said to be \emph{strongly orthogonal} if $eAe' = e'Ae = 0$.
\end{definition}

\begin{lemma} \label{lem:894}
Suppose $e,e' \in \Ai$ are strongly orthogonal. Let $M\in \cat{M}(A)_{e}$ and $N\in \cat{M}(A)_{e'}$. Then, $\Hom_{A}(M,N) = 0$.
\end{lemma}
\begin{proof}
Suppose $f\in \Hom_{A}(M,N)$. Then, $f(M) \in \cat{M}(A)_{e}\cap \cat{M}(A)_{e'}$, so by Lemma \ref{lem:891} we get that $f=0$.
\end{proof}

\begin{theorem} \label{prop:896}
Let $\mathcal{E} = \{e_{i}\}_{i\in I}$ be a family of left special idempotents in $A$. Then, the following are equivalent: 

\begin{enumerate}
\rmitem{a.} For every $M\in \cat{M}(A)$, the canonical morphism $\bigoplus_{i\in I} \Gamma_{e_{i}}(M) \to M$ is an isomorphism.
\rmitem{b.} The canonical functor $\pi_{\mathcal{E}}:\cat{M}(A) \to \prod_{i\in I} \cat{M}(A)_{e_{i}}$ induced by the functors $\Gamma_{e_{i}}: \cat{M}(A)\to \cat{M}(A)_{e_{i}}$ for $i\in I$ is an equivalence of categories.
\rmitem{c.} The following two conditions hold:
\begin{enumerate}
\rmitem{i} The family $\mathcal{E}$ is pairwise strongly orthogonal.
\rmitem{ii} The ideal $\sum_{i\in I} Ae_{i}A = A$.
\end{enumerate}
\end{enumerate}
\end{theorem}
\begin{proof}

We first show that $\pi_{\mathcal{E}}$ is essentially surjective if and only if $\mathcal{E}$ is pairwise strongly orthogonal.

First assume that $\mathcal{E}$ is pairwise strongly orthogonal, and let $(M_{i})_{i\in I}$ be an object of $\prod_{i\in I} \cat{M}(A)_{e_{i}}$. Consider $M:= \bigoplus_{i\in I} M_{i}$; $M\in \cat{M}(A)$. Since the distinct elements of $\mathcal{E}$ are strongly orthogonal to each other, $\Gamma_{e_{i}}(M) = M_{i}$. Thus $\pi_{\mathcal{E}}(M) = (M_{i})_{i\in I}$.

Conversely, suppose $\pi_{E}$ is essentially surjective. We wish to show that $\mathcal{E}$ is pairwise strongly orthogonal. Apropos, let $e_{i},e_{j}\in \mathcal{E}$, with $i\neq j$. Suppose $0\neq N\in \cat{M}(A)_{e_{i}}\bigcap \cat{M}(A)_{e_{j}}$. Consider the object $(N_{l})_{l\in I}$ with $N_{i} = N$ and $N_{k} = 0$ for all $k\neq i$. Suppose $(N_{l})_{l\in I}\cong \pi_{\mathcal{E}}(M)$ for some $M\in \cat{M}(A)$. Then, $\Gamma_{e_{i}}(M) = N$, while $\Gamma_{e_{j}}(M) = 0$. This is a contradiction. 

We now prove the theorem. If (a.) holds, then for any $M\in \cat{M}(A)$ and $e_{i}, e_{j}\in \mathcal{E}$ such that $i\neq j$, $\Gamma_{e_{i}}(M)\cap \Gamma_{e_{j}}(M) = 0$. Thus $\cat{M}(A)_{e_{i}}\cap \cat{M}(A)_{e_{j}} = 0$, and $e_{i}$ and $e_{j}$ are strongly orthogonal. Thus the functor $\pi_{\mathcal{E}}$ is essentially surjective; using Lemma \ref{lem:894}, we see that $\pi_{\mathcal{E}}$ is faithful. It is clear that (a.) forces $\pi_{\mathcal{E}}$ to be full. Thus (a.) implies (b.).

Now assume (b.) If $\pi_{\mathcal{E}}:\cat{M}(A) \to \prod_{i\in I} \cat{M}(A)_{e_{i}}$ is an equivalence of categories, it is essentially surjective. Thus  (i) holds. If $\sum_{i\in I} Ae_{i}A \neq A$, then $L: = A/\sum_{i\in I} Ae_{i}A$ is an object of $\cat{M}(A)$ such that $e_{i}L = 0$ for all $i\in I$. Then $\pi_{\mathcal{E}}(L) = 0$, which contradicts the fact that $\pi_{\mathcal{E}}$ is faithful. Thus (b.) implies (c.).

Assume (c.), and let $M\in \cat{M}(A)$. Since $M=AM$, $M = \sum_{i\in I} \Gamma_{e_{i}}(M)$. Since the elements of $\mathcal{E}$ are pairwise strongly orthogonal, this sum is direct. Thus we obtain (a.).
\end{proof}

\begin{definition} \label{def:895}
A family $\{e_{i}\}_{i\in I}$ of left special idempotents in $A$ which satisfies the equivalent conditions in Theorem \ref{prop:896} is said to be \emph{full}.
\end{definition}

\begin{corollary} \label{cor:897}
Let the  family $\mathcal{E} = \{e_{i}\}_{i\in I}$ be full. Then, every idempotent $e_{i}\in \mathcal{E}$ is left split. \qed
\end{corollary} 

\begin{theorem} \label{prop:898}
Let the family $\mathcal{E} = \{e_{i}\}_{i\in I}$ be full. Then, the canonical functor $\tilde{\pi}_{\mathcal{E}}:\cat{D}(A) \to \prod_{i\in I} \cat{D}(A)_{e_{i}}$ induced by the functors $\Gamma_{e_{i}}: \cat{D}(A)\to \cat{D}(A)_{e_{i}}$ for $i\in I$ is an equivalance of categories.
\end{theorem}
\begin{proof}
We only sketch the proof. It is clear from Theorem \ref{prop:896} that $\tilde{\pi}_{\mathcal{E}}$ is essentially surjective. We therefore only need to show that $\tilde{\pi}_{\mathcal{E}}$ is full and faithful. By Corollary \ref{cor:897}, each $e\in \mathcal{E}$ is left split. If $M$ and $N$ are objects in $\cat{D}(A)$, by Proposition \ref{prop:501} we can assume that both $M$ and $N$ are K-projective. Then, by Lemma \ref{l:078}, $\Gamma_{e}(M)$ and $\Gamma_{e}(N)$ are K-projective for every $e\in \mathcal{E}$; in particular, for every $e\in \mathcal{E}$ $\Hom_{\cat{D}(A)}(\Gamma_{e}(M), \Gamma_{e}(N)) = \Hom_{\cat{K}(A)}(\Gamma_{e}(M), \Gamma_{e}(N))$. The result now follows from Theorem \ref{prop:896}.
\end{proof}

\subsection{Miscellany} 
It would be interesting and useful to characterize left special and left split idempotents for a given idempotented algebra. This seems to be a difficult question in general: for instance, consider the case of $\mathcal H(G)$, as described in \S \ref{ss:red_grp}.

For simple idempotented algebras, however, the situation presents no difficulty.  
\begin{proposition}\label{prop_simple}
Let $A$ be a simple algebra. Then, every $e \in \Ai$ is left special and left split.
\end{proposition}
\begin{proof}
Suppose $e\in \Ai$. Since $A$ is simple, $A = AeA$. This implies that, for any $M\in \cat{M}(A)$, $AeM = AM = M$. Thus $\cat{M}(A)_{e} = \cat{M}(A)$. In particular, since $\cat{M}(A)_{e}$ is closed under subquotients, $e$ is left special; since $\cat{M}(A)_e = \cat{M}(A)$, $e$ is left split. 
\end{proof}

\begin{example}
Let $F$ be an arbitrary field and let $V$ be a vector space of countably infinite dimension over $F$. Suppose $\{x_{i}\}_{i\in \mathbb{N}}$ be a fixed ordered basis of $V$. Define $e_{n} \in \mathrm{End}_{F}(V)$ to be the idempotent linear transformation with $e_{n}(x_{i}) =x_{i}$ if $i\leq n$ and $e_{n}(x_{i})=0$ if $i>n$. Set $A_{n} = e_{n} \mathrm{End}_{F}(V) e_{n}$ and $A= \varinjlim_{n\in \mathbb{N}} A_{n}$. The algebras $A_{n}$ ($\cong M_{n}(F)$) and $A$ are simple idempotented algebras. Thus, by Proposition \ref{prop_simple}, every idempotent in these algebras is left special and left split.
\end{example}

The following result provides us with examples of idempotents that are both left special and left split, in another specific situation.
\begin{proposition}\label{prop_cen}
Let $e\in \Ai$ be a central idempotent. Then, $e$ is left special and left split.
\end{proposition}
\begin{proof}
Let $M\in \cat{M}(A)$. Since $e$ is central, $AeM = eM$. Therefore the category $\{M\in \cat{M}(A) \ \vert \ M = AeM\}$ is precisely the full subcategory of all $M\in \cat{M}(A)$ such that $M=eM$, and is clearly closed under subquotients. Thus $e$ is left special. Note that $\Gamma_{e}(M) = eM$, for any $M\in \cat{M}(A)$; in particular, $\Gamma_{e}$ is exact. By Proposition \ref{prop:adm_1} and  Lemma \ref{lem:04}, as a functor from $\cat{M}(A) \to \cat{M}(A)$, $\Gamma_{e}$ is its own right adjoint. The right adjoint of any exact functor preserves injective objects: by Lemma \ref{lem:07}, $\Gamma_{e}$ is left split. 
\end{proof}

\section{Path algebras} \label{ss:patha}
We now study left special and left split idempotents in path algebras. We begin by recalling some preliminaries.

Recall that a \emph{quiver} is a quadruple $Q: = (V,E,s,t)$, with sets $V$ and $E$ and functions $s,t: E\to V$. The elements of $V$ and $E$ are, respectively, called the \emph{vertices} and \emph{edges} of $Q$, and if $a$ is an edge, we respectively refer to $s(a)$ and $t(a)$ as the \emph{source} and \emph{target} of $a$. Let $E_{V}$ be the set of symbols $\{e_{v}\ \vert \  v\in V\}$; there is a canonical bijection between $E_{V}$ and $V$. A path in $Q$ is either a finite sequence of edges $a_{1},\ldots,a_{n}$ such that $t(a_{i}) = s(a_{i+1})$ (we will denote such a path by the concatenation $a_{n}\ldots a_{1}$), or an element of $E_{V}$. Let $P$ denote the set of all paths in $Q$. The elements of $P \setminus E_{V}$ are called \emph{non-trivial} paths, while those in $E_{V}$ are called $\emph{trivial}$. The path $e_{v}$ should be thought of as the path of length $0$ at the vertex $v$. 

There are obvious extensions of $s$ and $t$ to functions $s,t: P\to V$.

\begin{definition}
Let $Q$ be a quiver. The path algebra of $Q$, denoted by $\mathbb{K}Q$, is the free module on the set $P$, with multiplication defined as follows and then extended bilinearly: 
\begin{enumerate}
\item If $p$ and $q$ are paths with $t(q) \neq s(p)$, then $pq:= 0$.
\item If $p$ and $q$ are non-trivial paths with $t(q) = s(p)$, then $pq$ is defined as the concatenation of $p$ and $q$.
\item If $p$ is a path, then $e_{t(p)} p = p e_{s(p)} := p$.
\end{enumerate}
\end{definition}

For example, if we take $Q$ to be the quiver with a single vertex and $n$ distinct paths that start and end at that vertex, then $\mathbb{K}Q$ is the free associative $\mathbb{K}$-algebra in $n$ non-commuting variables.

\begin{remark}
For every finite subset $S\subseteq V$ the element $e_{S}:= \sum_{v\in S} e_{v}$ is an idempotent in the algebra $\mathbb{K}Q$. It follows that this algebra is always idempotented: it is unital precisely when the vertex set $V$ is finite, in which case $e_{V}$ is the unit.
\end{remark}

\begin{definition}
Let $Q$ be a quiver, with vertex set $V$. A finite subset $S\subseteq V$ is said to be \emph{left closed} if whenever $v\in S$ and $a$ is an edge with $s(a) = v$, then $t(a)\in S$. Similarly, a finite subset $S\subseteq V$ is said to be \emph{right closed} if whenever $v\in S$ and $a$ is an edge with $t(a) = v$, then $s(a) = v$.
\end{definition}

\begin{remark}\label{rem_lr_closed}
Note that a finite set $S\subseteq V$ is left and right closed precisely when it is a union of connected components of the underlying directed graph $(V,E)$ of $Q$.
\end{remark}

\begin{proposition} \label{prop_left_clo1}
Let $Q$ be a quiver, with vertex set $V$. Let $e\in \mathbb{K}Q$. Then, $e$ is left special if and only if there exists 
\begin{enumerate}
\item A left closed subset $S\subseteq V$, 
\item An $S$-indexed family $\{\lambda_{v}\}_{v\in S}$ of non-zero elements of $\mathbb{K}_{\mathrm{idm}}$, such that $\mathbb{K}\lambda_{v}\subseteq \mathbb{K}\lambda_{v'}$ whenever there is a path $p$ with $s(p)=v$ and $t(p)=v'$,
\item A collection of distinct paths $p_{1},\ldots, p_{n}$, with $t(p_{i})\in S$ for all $i\in \{1,\ldots,n\}$, 
\item A collection of non-zero elements $\kappa_{1},\ldots,\kappa_{n}$ of $\mathbb{K}$ with $\lambda_{t(p_{i})} \kappa_{i}= \kappa_{i}$ for all $i\in \{1,\ldots,n\}$, and $\lambda_{s(p_{i})}\kappa_{i} = 0$ for all $i\in \{1,\ldots,n\}$ with $s(p_{i})\in S$,
\end{enumerate}
 such that $e = \sum_{v\in S} \lambda_{v}e_{v} + \sum_{i=1}^{n} \kappa_{i}p_{i}$.
\begin{proof}
Suppose $e = \sum_{v\in S} \lambda_{v}e_{v} + \sum_{i=1}^{n} \kappa_{i}p_{i}$, where the data in this expression is as in the hypothesis of the proposition. It is an easy check that $e = e^{2}$. Let $M\in \cat{M}(\mathbb{K}Q)$, and suppose $M = \mathbb{K}QeM$; since $S$ is left closed, $e_{S}M = M$. It follows that for $v\in S$, \[e_{v}M = (\sum_{\mathclap{\substack{p\in P \\ t(p) = v}}}\mathbb{K}p)(\sum_{v\in S} \lambda_{v}e_{v} + \sum_{i=1}^{n} \kappa_{i}p_{i})M = \sum_{\mathclap{\substack{p\in P \\ s(p)\in S \\ t(p)=v}}} \mathbb{K}\lambda_{s(p)}pM + \sum_{{\substack{i \in \{1,\ldots, n\} \\ s(p_{i})\in S \\ p\in P \\ t(p)=v \\ s(p) = t(p_{i})}}} \mathbb{K}\kappa_{i}pp_{i}M.\] The hypothesis implies that each $\lambda_{s(p)}\in \mathbb{K}\lambda_{v}$ and that $\lambda_{v}\kappa_{i} = \kappa_{i}$ for all $p$ and $i$ occuring in the above sum. Thus $\lambda_{v}e_{v}M = e_{v}M$ for all $v\in S$. Together with $e_{S}M=M$, this implies that $e'M=M$, where $e' = \sum_{v\in S} \lambda_{v}e_{v}$. However, the hypothesis implies that $e'e=e$, while $ee' = e'$. This in turn implies that $e'M = M$ if and only if $eM=M$.  Thus in summary, $M = \mathbb{K}QeM$ if and only if $eM = M$. It is obvious that this condition is inherited by submodules. Thus $e$ is left special.

Now we show the converse. Let $e\in \mathbb{K}Q$, and suppose $e$ is left special. Let $S\subseteq V$ be the collection of all $v\in V$ such that $e_{v}$ appears as a non-zero term when $e$ is written in terms of the canonical generating set of $\mathbb{K}Q$; thus, there exist $\{\lambda_{v}\}_{v\in S}$,  $p_{1},\ldots,p_{n}$, and $\kappa_{1},\ldots,\kappa_{n}$ such that $e = \sum_{v\in S} \lambda_{v}e_{v} + \sum_{i=1}^{n} \kappa_{i}p_{i}$, where the $\lambda_{v}$ and $\kappa_{i}$ are elements of $\mathbb{K}$, and the $p_{i}$ are distinct non-trivial paths in $Q$. A straightforward multiplication shows that each $\lambda_{v}$ is idempotent. 

Suppose $S$ is not left closed. Then, there is an edge of $Q$, say $a'$, such that $s(a')$ belongs to $S$ but $t(a')$ does not. Define $M = \mathbb{K}\lambda_{s(a')}$, and consider $M\times M$. We give $M\times M$ the structure of a non-degenerate $\mathbb{K}Q$-module as follows. Let $e_{s(a')}$ act via the map $(x,y)\mapsto (x,0)$, $e_{t(a')}$ act via the map $(x,y)\mapsto (0,y)$, and $a'$ act via the map $(x,y)\mapsto (0,x)$. Define the action of all other edges and trivial paths to be $0$. It can be checked that this extends naturally to give $M\times M$ the structure of a $\mathbb{K}Q$ module: call this module $N$. Since $e_{t(a')} + e_{s(a')}$ fixes $N$, $N$ is non-degenerate. 
Now $e_{s(a')}eN$ is the $\mathbb{K}$-submodule $M\times 0$; it follows that $\mathbb{K}QeN = N$. However, it is easily checked that $L:=0\times M$ is $\mathbb{K}Q$-submodule of $N$, and that $\mathbb{K}QeL = 0$. This is a contradiction. Thus $S$ is left closed.

A similar argument gives the condition on idempotents. Since $S$ is left closed, it is enough to show that for any edge $a$ with $t(a), s(a)\in S$, $\mathbb{K}\lambda_{s(a)}\subseteq \mathbb{K}\lambda_{t(a)}$; we can assume that $s(a)\neq t(a)$. Define $M' = \mathbb{K}\lambda_{s(a)}$, and construct a $\mathbb{K}Q$-module $N'$ as we did above for $M$ and $N$, but with the edge $a$ playing the role that $a'$ played in the previous construction. As before, $\mathbb{K}QeN'=N'$, and $L':=0\times M'$ is $\mathbb{K}Q$-submodule of $N'$. Since $e$ is assumed left special, $L' = \mathbb{K}QeL'$; by construction $\mathbb{K}QeL' = \lambda_{t(a)}e_{t(a)}L'$. In particular, $L'$ is fixed by $\lambda_{t(a)}$. Since $M' = \mathbb{K}\lambda_{s(a)}$, this implies that $\lambda_{s(a)}$ is fixed by $\lambda_{t(a)}$, or equivalently that $\mathbb{K}\lambda_{s(a)}\subseteq \mathbb{K}\lambda_{t(a)}$. 

Let $P''$ denote those paths in $P': =\{p_{i}\mid i=1,\dots, n\}$ which cannot be written as the concatenation of shorter non-trivial paths in $P'$; note that every path in $P'\setminus P''$ is a concatenation of paths in $P'$. Then, $e = \sum_{v\in S} \lambda_{v}e_{v} + \sum_{p_{i}\in P''} \kappa_{i}p_{i} + \sum_{p_{i}\in P'\setminus P''} \kappa_{i}p_{i}.$ Let $p_{i}\in P''$. If  $s(p_{i}) \in S$, then since $S$ is left closed $t(p_{i})\in S$. Since $e$ is idempotent, by squaring and comparing both sides of the preceding equation we see that $\kappa_{i}p_{i} = \lambda_{t(p_{i})}\kappa_{i}e_{t(p_{i})}p_{i} + \lambda_{s(p_{i})}\kappa_{i}p_{i}e_{s(p_{i})}$. Therefore $\kappa_{i} = (\lambda_{s(p_{i})} + \lambda_{t(p_{i})})\kappa_{i}$. However, by what we have shown earlier in this proof, in this situation $\lambda_{s(p_{i})} \in \mathbb{K}\lambda_{t(p_{i})}$: thus, $\lambda_{t(p_{i})}\kappa_{i} = \kappa_{i}$, and $\lambda_{s(p_{i})} \kappa_{i} = 0$.
 If $s(p_{i})\notin S$, a similar, but easier, comparison shows that $\lambda_{t(p_{i})} \kappa_{i}= \kappa_{i}$. Finally, suppose that $p_{i} \in P'\setminus P''.$ If $s(p_{i})\in S$, then we can write $\kappa_{i}p_{i} = \lambda_{t(p_{i})}\kappa_{i}e_{t(p_{i})}p_{i} + \lambda_{s(p_{i})}\kappa_{i}p_{i}e_{s(p_{i})} + \sum_{j,k} \kappa_{j}\kappa_{k} p_{j}p_{k}$, where the $p_{j}$ and $p_{k}$ are shorter concatenable non-trivial paths such that $p_{j}p_{k} = p_{i}$. By induction, $\lambda_{t(p_{k})}\kappa_{k} = \kappa_{k}$, while $\lambda_{s(p_{j})}\kappa_{j} = 0$. Thus $\kappa_{j}\kappa_{k}=0$. It follows that $\kappa_{i}p_{i} = \lambda_{t(p_{i})}\kappa_{i}e_{t(p_{i})}p_{i} + \lambda_{s(p_{i})}\kappa_{i}p_{i}e_{s(p_{i})}$, and we can now argue as before. The case where $s(p_{i})\notin S$ is similar. This completes the proof.
\end{proof}
\end{proposition}

\begin{corollary} \label{cor112211}
Let $Q$ be a quiver, with vertex set $V$. Let $S$ be a finite subset of $V$. Then, the idempotent $e_{S}$ is left special if and only if $S$ is left closed. \qed
\end{corollary}

\begin{definition} \label{dfn334455}
Let $Q$ be a quiver, and let $e\in \mathbb{K}Q$ be left special. The expression  $e = \sum_{v\in S} \lambda_{v}e_{v} + \sum_{i=1}^{n} \kappa_{i}p_{i}$, where this data is as in Proposition \ref{prop_left_clo1}, is called the \emph{standard form} of $e$.
\end{definition}

\begin{proposition} \label{prop_right_clo2}
Let $Q$ be a quiver, with vertex set $V$. Let $e \in \mathbb{K}Q$ be left special, with standard form $e = \sum_{v\in S} \lambda_{v}e_{v} + \sum_{i=1}^{n} \kappa_{i}p_{i}$. Then $e$ is left split if and only if $S$ is right closed, and $\lambda_{v} = \lambda_{v'}$ for all $v$ and $v'$ in $S$ lying in the same connected component of $V$. 
\end{proposition}
\begin{proof}
If $S$ is right closed and $\lambda_{v} = \lambda_{v'}$ for all $v$ and $v'$ in the same connected component of $V$, then there are no non-trivial paths in the standard form of $e$, and by Remark \ref{rem_lr_closed} $e$ is a central element of $\mathbb{K}Q$. By Proposition \ref{prop_cen}, it is left split. 

Conversely, let $e$ be left split. Suppose $S$ is not right closed. There is an edge $a$ such that $t(a)$ lies in $S$ while $s(a)$ does not. The algebra $\mathbb{K}Q$ admits a decomposition $\mathbb{K}Q = \mathbb{K}Qe\mathbb{K}Q  \oplus A'$, where $A'$ is some $\mathbb{K}Q$-submodule. It is easy to check that in this situation, $\lambda_{t(a)}e_{s(a)} \in A'$. This is a contradiction, as $a \lambda_{t(a)}e_{s(a)} =  \lambda_{t(a)}a =  \lambda_{t(a)}e_{t(a)}a \in  \mathbb{K}Qe\mathbb{K}Q$. Thus $S$ is right closed. 

We are left to show the condition on idempotents. Let $e' := \sum_{v\in S} \lambda_{v}e_{v}$. It is enough to show that for any edge $a$ with $t(a), s(a)\in S$, $\lambda_{t(a)} = \lambda_{s(a)}$. By Proposition \ref{prop_left_clo1}, we know that $\mathbb{K}\lambda_{s(a)}\subseteq \mathbb{K}\lambda_{t(a)}$.
Consider the element $(\lambda_{t(a)} - \lambda_{s(a)})e_{s(a)}$. Since $e$ is left split, there exist $x\in \mathbb{K}Qe\mathbb{K}Q$ and $y \in A'$ such that $(\lambda_{t(a)} - \lambda_{s(a)})e_{s(a)} = x+y.$ Then $e'(\lambda_{t(a)} - \lambda_{s(a)})e_{s(a)} = e'x + e'y$. We have $e'(\lambda_{t(a)} - \lambda_{s(a)})e_{s(a)} = (\lambda_{t(a)} - \lambda_{s(a)})\lambda_{s(a)}e_{s(a)} = 0$, since $\lambda_{t(a)}$ fixes $\lambda_{s(a)}$. Thus $e'x=e'y= 0$. However, by the proof of Proposition \ref{prop_left_clo1}, $e'x=ex=x$. Thus $x=0$, and $(\lambda_{t(a)} - \lambda_{s(a)})e_{s(a)} \in A'$. This implies that $\lambda_{t(a)}a(\lambda_{t(a)} - \lambda_{s(a)})e_{s(a)} \in A'\cap \mathbb{K}Qe\mathbb{K}Q = 0$, and that \[ 0= \lambda_{t(a)}a(\lambda_{t(a)} - \lambda_{s(a)})e_{s(a)} = \lambda_{t(a)}(\lambda_{t(a)} - \lambda_{s(a)})a.\] Thus $\lambda_{t(a)} = \lambda_{t(a)}\lambda_{s(a)}$. Since we know that $\lambda_{s(a)} = \lambda_{t(a)}\lambda_{s(a)}$, we are done. 
\end{proof}

\begin{corollary} \label{cor:887766}
Let $Q$ be a quiver, with vertex set $V$.  Let $S$ be a left closed subset of $V$. Then, $e_{S}$ is left split if and only if $S$ is right closed. \qed
\end{corollary}

\begin{example}
Suppose $Q$ is the quiver $v_{1} \xrightarrow{p} v_{2}$ and let $S=\{v_{2}\}$. By Corollaries \ref{cor112211} and \ref{cor:887766}, it follows that the idempotent $e_{S}$ is left special but not left split.
\end{example}

\begin{proposition} \label{prop_dec}
Suppose $\mathbb{K}$ has no idempotents apart from $0$ and $1$. Let $Q$ be a quiver, with vertex set $V$. A family of left special idempotents $\mathcal{E} = \{e_{i}\}_{i\in I}$ is full if and only if there is a family of left closed subsets of $V$, $\mathcal{S} := \{S_{i}\}_{i\in I}$, with $S_{i}\cap S_{j} = \emptyset$ for $i\neq j$ and $V = \bigcup_{i\in I} S_{i}$, such that $e_{i} = e_{S_{i}}$ for all $i\in I$.
\end{proposition}
\begin{proof}
First, consider a family $\mathcal{S} := \{S_{i}\}_{i\in I}$ of left closed subsets of $V$, with $S_{i}\cap S_{j} = \emptyset$ for $i\neq j$ and $V = \bigcup_{i\in I} S_{i}$. By Corollary \ref{cor112211}, each $e_{S_{i}}$ is left special. Let $S_{i}$ and $S_{j}$ belong to $\mathcal{S}$. $e_{S_{j}}\mathbb{K}Qe_{S_{i}}$ is the module generated by all paths in $Q$ which start at a vertex in $S_{i}$ and end at a vertex in $S_{j}$. Since $S_{i}$ is left closed, such a path must have its target lie in $S_{i}\cap S_{j}$. We conclude that $e_{S_{j}}\mathbb{K}Qe_{S_{i}} = 0$. Thus, for $i\neq j$ the idempotents $e_{S_{i}}$ and $e_{S_{j}}$ are strongly orthogonal. $\mathbb{K}Qe_{S_{i}}\mathbb{K}Q$ is the module generated by all paths which pass through a vertex in $S_{i}$. As $S_{i}$ is left closed, $\mathbb{K}Qe_{S_{i}}\mathbb{K}Q = e_{S_{i}}\mathbb{K}Q$. Since the target of  every path in $\mathbb{K}Q$ lies in some $S_{i}\in \mathcal{S}$, it follows that $\{e_{S_{i}}\}_{i\in I}$ is full.

For the converse, suppose  $\mathcal{E} = \{e_{i}\}_{i\in I}$ is a full family of left special idempotents. By Corollary \ref{cor:897}, every idempotent in $\mathcal{E}$ is left split. By Propositions \ref{prop_left_clo1} and \ref{prop_right_clo2}, for every $e_{i}\in \mathcal{E}$ there is a right and left closed subset $S_{i}$ of $V$ such that $e_{i} = e_{S_{i}}.$ Since the elements of $\mathcal{E}$ are pairwise strongly orthogonal, for $i\neq j$, we get that $S_{I}\cap S_{j} = \emptyset$. It is also straightforward to see that $V = \bigcup_{i\in I} S_{i}$.
\end{proof}

\begin{remark}
Similar arguments provide a variant of Proposition \ref{prop_dec} without the additional condition on idempotents in $\mathbb{K}$, which we omit here for the sake of brevity.
\end{remark}

\end{document}